\documentclass[11 pt]{amsart}

\usepackage[latin1]{inputenc}
\usepackage{amsfonts,amssymb,amsmath,amsthm}
\usepackage[headinclude,DIV13]{typearea}
\usepackage{hyperref}

\usepackage{geometry}
\usepackage{graphicx, psfrag}

 \usepackage{amsmath,amsfonts,amssymb}
 \usepackage{color}
 \usepackage{setspace}

\newtheorem{rem}{Remark}
\newtheorem{lem}{Lemma}[section]
\newtheorem{pro}{Proposition}[section]
\newtheorem{defi}{Definition}[section]

\newtheorem{theo}{Theorem}

\newtheorem{cor}{Corollary}[section]
\newtheorem{ass}{Assumption}

\theoremstyle{definition} 
\theoremstyle{definition} 

\renewcommand{\P}{\mathbb{P}}
\newcommand{\R}{\mathbb{R}}

\newcommand{\E}{\mathbb{E}}

\newcommand{\N}{\mathbb{N}}
\newcommand{\Z}{\mathbb{Z}}

\newcommand{\eps}{\varepsilon}

\numberwithin{equation}{section}

\begin{document}

\title[The effect of recurrent mutations 
on genetic diversity]
{The effect of recurrent mutations 
on genetic diversity
in a large population of varying size}

\author{Charline Smadi}
\address{Irstea, UR LISC, Laboratoire d'Ing\'enierie des Syst\`emes Complexes, 9 avenue Blaise Pascal-CS 20085, 63178 Aubi\`ere, France and 
Department of Statistics, University of Oxford, 1 South Parks Road, Oxford OX1 3TG, UK}
\email{charline.smadi@polytechnique.edu}

\keywords{ Eco-evolution; Birth and death process with immigration; Selective sweep; Coupling; competitive Lotka-Volterra system with mutations}
\subjclass[2010]{92D25, 60J80, 60J27, 92D15, 60F15, 37N25.}

\begin{abstract}
Recurrent mutations are a common phenomenon in population genetics. They may 
be at the origin of the fixation of a new genotype, if they give a phenotypic advantage to the carriers of the new mutation.
In this paper, we are interested in the genetic signature left by a selective sweep induced by recurrent mutations at a given locus 
from an allele $A$ to an allele $a$, depending on the mutation frequency. We distinguish three possible scales for the mutation 
probability per reproductive event, which entail distinct genetic signatures. Besides, we study the hydrodynamic limit of the $A$- 
and $a$-population size dynamics when mutations are frequent, and find non trivial equilibria leading to several possible patterns of polymorphism.
\end{abstract}

\maketitle

\section*{Introduction}

Recurrent mutations are a common phenomenon in population genetics \cite{haldane1933part}. They may cause diseases 
\cite{shore2006recurrent,erkko2007recurrent,risheg2007recurrent}, or reduce the 
fitness of an individual in a given environment \cite{repping2003polymorphism}. 
They may also be at the origin of the fixation of a new 
genotype, if they give a phenotypic advantage to the carriers of the new mutation 
\cite{nohturfft1996recurrent,cresko2004parallel,chan2010adaptive,skuce2010benzimidazole}.
These selective sweeps are then called 
"soft sweeps", as multiple copies of the selected allele contribute to a substitution \cite{hermisson2005soft}.
More precisely, we call them soft selective sweeps from recurrent mutations, to make the difference 
with soft selective sweeps from standing variation, when preexisting alleles become advantageous
after an environmental change.
The last possibility is a hard selective sweep, when selection acts 
on a new mutation. 
In mathematical works, hard selective sweeps were until recently the only type of adaptation considered. 
But soft selective sweeps allow a faster adaptation to novel
environments, and their importance is growing in empirical and theoretical studies.
In particular Messer and Petrov \cite{messer2013population} review a lot of evidence,
from individual case studies as well as from genome-wide scans, that soft sweeps (from standing
variation and from recurrent mutations) are common in a broad range of organisms. These distinct 
selective sweeps entail different footprints in DNA in the vicinity of the newly fixed allele, called genetic signatures,
and the accumulation of genetic data available allows one to detect these signatures in current
populations (see for instance \cite{nohturfft1996recurrent,tishkoff2007convergent,nair2007recurrent,Scheinfeldt2009} or references in \cite{wilson2014soft}). 
To do this in an effective way, it is necessary to identify accurately the signatures left by these
different modes of adaptation.

In this paper, we are interested in the genetic signature left by a selective sweep induced by recurrent mutations at a given locus
from an allele $A$ to an allele $a$, depending on the mutation frequency. 
We distinguish three possible scales for the mutation probability per reproductive event, which 
entail distinct genetic signatures.
For the sake of completness, we also consider back mutations, from $a$ to $A$.
Besides, we study the hydrodynamic limit of the $A$- and $a$-population size dynamics when mutations are likely to happen 
(mutation probabilities of order $1$), and find 
non trivial equilibria leading to several possible patterns of polymorphism.
Note that cases of polymorphism maintained by recurrent mutations to a deleterious allele have been reported in $Y$ 
chromosome for instance \cite{repping2003polymorphism}.
We consider an asexual haploid population of varying size modeled by a birth and
death process with density dependent competition. Each individual's ability to survive and reproduce 
depends on its own genotype and on the population state. More precisely, each individual
is characterized by some ecological parameters: birth rate, intrinsic death rate and competition
kernel describing the competition with other individuals depending on their genotype. The differential 
reproductive success of individuals generated by their interactions entails progressive
variations in the number of individuals carrying a given genotype. This process, called natural
selection, is a key mechanism of evolution. Such an eco-evolutionary approach has been introduced 
by Metz and coauthors in \cite{metz1996adaptive} and made rigorous in a probabilistic setting in the seminal paper of Fournier and
Méléard \cite{fournier2004microscopic}. Then it has been developed by Champagnat, Méléard and coauthors (see 
\cite{champagnat2006microscopic,champagnat2011polymorphic,champagnat2014adaptation}
and references therein). This approach has already been used to study the signature left by  hard selective sweeps 
and soft selective sweeps from standing variation in \cite{smadi2014eco,rebekka2015genealogies}.

The first theoretical studies of the signature left by selective sweep from recurrent mutations are due to 
Pennings and Hermisson \cite{pennings2006soft,pennings2006soft2}. They use a Wright-Fisher model with selection 
and mutations. They derive the probability of a soft sweep for a sample of the population (i.e. the probability to 
have different mutational origins in the sample) and describe how a soft sweep from recurrent
mutations affects a neutral locus at some recombinational
distance from the selected locus and which tests can be
employed to detect soft sweeps. Following these works, Hermisson and Pfaffelhuber \cite{hermisson2008pattern} 
rigorously study sweeps from recurrent mutations, providing in particular the duration of the sweep, and 
an approximation of the genealogies by a marked Yule process with immigration. 
They use this latter to improve analytical approximations for the expected heterozygosity at
the neutral locus at the time of fixation of the beneficial allele.
Finally, Pokalyuk \cite{pokalyuk2012effect} studies the linkage disequilibrium of two neutral 
loci in the vicinity of the selected allele.
Unlike our model, in all these works, the population size is
constant and the individuals' "selective value" does not depend on the population state, but only
on individuals' genotype. Moreover, back mutations are not considered.

The structure of the paper is the following. In Section \ref{sectionmodel} we describe the model and review some
results of \cite{champagnat2006microscopic} about the population process when there is only 
one mutation from $A$ to $a$. 
In Section \ref{sectionresults} we present the main results of the paper, and comment their biological implications in Section \ref{Discussion}.
In Section \ref{sectionpoisson} we introduce a Poisson representation of the population process, 
as well as couplings of population sizes with birth and death processes without competition, 
during the beginning of the sweep. They are key tools in the different proofs. 
Sections \ref{sectionproofduration} to \ref{proof_prop_conv} are devoted to the proofs.
In Section \ref{sectionillustr} we illustrate some of the results.
Finally in the Appendix
we state technical results.

\section{Model} \label{sectionmodel}

We consider an haploid asexual population and focus on one locus with alleles in 
$\mathcal{A}:=\{A,a\}$.
We are interested in the effect of recurrent mutations on the genetic diversity at this locus after a selective sweep.
We assume that mutations occur at birth. 
Let $\bar{\alpha}$ denote the complementary type of $\alpha$ in $\mathcal{A}$. Mutations occur as follows:
when an $\alpha$ individual gives birth, the newborn is of type $\alpha$ with a probability $1-\mu_K^{\alpha \bar{\alpha}}$,
and of type $\bar{\alpha}$ with a probability $\mu_K^{\alpha \bar{\alpha}}$.
Here the parameter $K$ is the environment's carrying capacity, which is a measure of the maximal population size that the environment can sustain for a long time.
We will see that the genetic diversity at the end of the sweep is very dependent on the scaling of $\mu_K$ with $K$.
More precisely, we will consider the four following possibilities for $\alpha \in \mathcal{A}$:
\begin{ass}\label{assmuK1}
$$ \mu_K^{\alpha \bar{\alpha}} =o\left( \frac{1}{K \log K}\right),  \quad K \to \infty. $$
\end{ass} 
\begin{ass}\label{assmuK2}
$$ \mu_K^{\alpha \bar{\alpha}} \sim \lambda^{\alpha \bar{\alpha}}K^{-1},  \quad K \to \infty, 
\quad \lambda^{\alpha \bar{\alpha}} \in \R_+^*. $$
\end{ass}  
\begin{ass}\label{assmuK3}
$$ \mu_K^{\alpha \bar{\alpha}} \sim \lambda^{\alpha \bar{\alpha}}K^{-1+\beta},  \quad K \to \infty,
\quad \lambda^{\alpha \bar{\alpha}} \in \R_+^*, \quad \beta \in (0,1) . $$
\end{ass} 
\begin{ass}\label{assmuK4}
$$ \mu_K^{\alpha \bar{\alpha}} \sim \lambda^{\alpha \bar{\alpha}},  \quad K \to \infty, \quad \lambda^{\alpha \bar{\alpha}} \in (0,1). $$
\end{ass} 

Assumption \ref{assmuK1} is a weak mutation limit often considered in adaptive dynamics models 
\cite{diekmann2004beginner,champagnat2006microscopic}. It leads to hard sweeps: 
a single mutant is at the origin of the fixation of a new allele.
Under Assumption \ref{assmuK2}, the mutation rate at the population scale is of order one.
This is the usual scaling for diffusion approximations and has been widely studied with the constant population assumption 
(see for instance \cite{kimura1978stepwise,griffiths1980lines,nagylaki1983evolution}). In the third regime, the product of mutation rate and population 
size increases as a power law. This may be a less interesting model from a biological point of view but it has the virtue of completing the pattern by 
linking the second and fourth regimes. Under the last assumption, mutations are frequent and the population is large.
The highest mutation rates have so far been reported in viruses: for example
$1.5 \ 10^{-3}$ mutations per
nucleotide, per genomic replication in
the single-stranded RNA phage $Q\beta 10$ \cite{drake1993rates}, $2.5 \ 10^{-3}$ mutations per site and replication cycle in Chrysanthemum
chlorotic mottle viroid \cite{gago2009extremely},
or  $4.1 \ 10^{-3}$ mutations per base per cell in HIV-1 virus \cite{cuevas2015extremely}.
Moreover, existing measures of mutation rate focus only on mean rates \cite{duffy2008rates}. 
In particular, some sites may have a higher mutation rate that what is measured. 
Besides, several mutations may generate the same phenotypic change \cite{van2013convergent}, leading to higher effective mutation rates.
Viruses can also have very high population sizes (for example up to $10^{12}$ RNA viruses in an organism \cite{moya2004population}), 
and it seems reasonable to consider such sizes as infinite as an approximation.
Hence the parameter range covered by Assumption \ref{assmuK4} could be relevant for such viruses. 
As we will see in the sequel, this is the 
only regime where back mutations (from $a$ to $A$) have a macroscopic effect on the pattern of allelic diversity (see Section \ref{backmut} for more details).

Under Assumption \ref{assmuK1}, the allele dynamics has already been studied by Champagnat in \cite{champagnat2006microscopic}. In this case, a favorable mutant $a$ (see condition \eqref{condfitnesses}) 
has time to fixate before the occurrence of a new mutation, 
and all the $a$ individuals are descended from the same mutant. Under Assumption \ref{assmuK2}, several mutant populations are 
growing simultaneously, and we can approximate the diversity of 
$a$ individuals at the end of the sweep (see Theorem \ref{theoprecise}). In the third case, the $a$ mutants are so numerous that with a probability close to one, two individuals sampled uniformly
at the end of the sweep will be descended from two different $a$ mutants.
Finally, under Assumption \ref{assmuK4}, the distribution of the alleles $a$ and $A$ in the population 
converges to a deterministic limit distribution (see Theorem \ref{theo_high_mut}).

Let us denote by $f_\alpha$ the fertility of an individual with type $\alpha$. The birth rate at the population level of individuals of 
type $\alpha \in \mathcal{A}$ is:
\begin{equation}\label{birth_rate} b^K_{\alpha}(n)=(1-\mu_K^{\alpha \bar{\alpha}})f_\alpha n_{\alpha} +
 \mu_K^{\bar{\alpha}\alpha } f_{\bar{\alpha}} n_{\bar{\alpha}}, \end{equation}
where $n_{\alpha }$ denotes the current number of $\alpha $-individuals 
 and 
$n=(n_{\alpha }, \alpha \in \mathcal{A})$
is the current state of the population.
The first term corresponds to clonal births, and the second to births with mutation.
An $\alpha$-individual can die either from a natural death (rate $D_\alpha$), or from type-dependent competition:
 the parameter $C_{\alpha,\alpha'}$ models the impact of an 
individual of type $\alpha'$ on an individual of type $\alpha$, where $(\alpha,\alpha') \in \mathcal{A}^2$. 
The strength of the competition also depends on
the carrying capacity $K$.
This 
results in the total death rate of individuals carrying the allele $\alpha  \in \mathcal{A}$:
\begin{equation}\label{deathrate}
 d^{K}_{\alpha }(n) = D_\alpha^K(n)n_\alpha=\left( D_{\alpha} +  \frac{C_{\alpha,A}}{K}n_{A} + \frac{C_{\alpha,a}}{K} n_{a} \right) n_\alpha.
\end{equation}

Hence the population process
\begin{align*}
N^K= (N^{K}(t), t \geq 0)=\Big((N_{A}^{K}(t),N_{a}^{K}(t)), t \geq 0\Big),
\end{align*}
where $N_{\alpha }^{K}(t)$ denotes the number of $\alpha $-individuals 
at time $t$, is a nonlinear multitype birth and death process with rates given in \eqref{birth_rate} and \eqref{deathrate}.

As a quantity summarizing the advantage or disadvantage of a mutant with allele type $\alpha$ in an $\bar{\alpha}$-population at equilibrium, we introduce the so-called invasion fitness 
$S_{\alpha \bar{\alpha}}$ through
\begin{equation} \label{deffitinv}
 S_{\alpha \bar{\alpha}} := f_{\alpha} -D_{\alpha} - C_{\alpha,\bar{\alpha}}\bar{n}_{\bar{\alpha}},
\end{equation}
where the equilibrium density $\bar{n}_{{\alpha}}$ is defined by
\begin{equation} \label{defbarn}
\bar{n}_{{\alpha}}: =\frac{f_{\alpha} -D_{\alpha}}{C_{\alpha,{\alpha}}}.
\end{equation}
The role of the invasion fitness $S_{\alpha \bar{\alpha}}$ and the 
definition of the equilibrium density $\bar{n}_{{\alpha}}$ follow from the properties of the two-dimensional competitive Lotka-Volterra 
system:
\begin{equation} \label{S}
 \dot{n}_\alpha^{(z)}=(f_\alpha-D_\alpha-C_{\alpha,A}n_A^{(z)}-C_{\alpha,a}n_a^{(z)})n_\alpha^{(z)},\quad 
z \in \R_+^\mathcal{A}, \quad
n_\alpha^{(z)}(0)=z_\alpha,\quad  \alpha \in \mathcal{A}.
 \end{equation}
If we assume 
\begin{equation}\label{condfitnesses}
 \bar{n}_{A}>0,\quad \bar{n}_{a}>0,\quad \text{and} \quad S_{Aa}<0<S_{aA},
\end{equation}
then $\bar{n}_{\alpha}$ is the equilibrium density of a monomorphic $\alpha$-population and 
the system \eqref{S} has a unique stable equilibrium $(0,\bar{n}_a)$ and two unstable steady states $(\bar{n}_A,0)$ and $(0,0)$.
Thanks to Theorem 2.1 p. 456 in \cite{ethiermarkov} we can prove that 
if $N_A^K(0)$ and $N_a^K(0)$ are of order $K$, Assumption \ref{assmuK1}, \ref{assmuK2} or \ref{assmuK3} holds, and $K$ is large, the rescaled process $(N_A^K/K,N_a^K/K)$ is very close to the solution of 
\eqref{S} during any finite time interval.
The invasion 
fitness $S_{aA}$ corresponds to the initial growth rate of the $a$-population, when the first $a$-mutant appears in a monomorphic 
population of individuals $A$ at their equilibrium size $\bar{n}_AK$, when we do not take into account the mutations.
Hence the dynamics of the allele $a$ is very dependent on the properties of the system \eqref{S}.\\

In order to study the process in detail and discriminate between distinct haplotypes
(i.e. the genetic materials transmitted by the successive mutants generated from $A$), we make a distinction between the descendants 
of the successive mutants generated from $A$.
Let $T_i$ denote the time of birth of the $i$th mutant of type $a$, and $N_i^K(t)$ the size of its descendance at time $t$. In particular, 
$$N_i^K(t)=0 \quad \forall t< T_i \quad \text{and} \quad N^K_a(t)= \underset{i \in \N }{ \sum } N_i^K(t), \quad \forall t \geq 0.$$
The state of the population process at time $t\geq 0$ is now represented by the vector 
$$ (N_A^K(t),N_1^K(t),N_2^K(t),... ). $$
Note that this representation is well defined as the number of mutations is almost surely finite in every time interval 
under our assumptions.

\section{results}\label{sectionresults}

For the sake of simplicity, we assume that at time $0$, the $A$-population is at its monomorphic equilibrium and that the first mutant has just appeared, that is to say 
\begin{equation}\label{cond_ini}
 N^K(0)= (\lfloor \bar{n}_A K \rfloor, 1).
\end{equation}

We define the end of the sweep $T_F^K$ as the time at which all ancestral $A$ type individuals have died out, i.e. all remaining $A$-individuals
are descendants of an $a$-individual via a back mutation.

Let us now state a result on the selective sweep duration, which makes precise how the sweep duration decreases 
when the mutation probability increases:

\begin{theo} \label{theoduration}
Assume that \eqref{condfitnesses} and \eqref{cond_ini} hold. Introduce $(\beta_1, \beta_2, \beta_3):=(0,0,\beta)$. 
Then under Assumption $i, \ i \in \{1,2,3\}$,
$$ \underset{K \to \infty}{\lim}\P(T_F^K<\infty)=1, $$
and
  $$ \frac{T_F^K}{\log K}\left(\frac{(1-\beta_i)}{S_{aA}}+ \frac{1}{|S_{Aa}|} \right)^{-1} \underset{K \to \infty}{\to} 1, \quad  \text{in probability}. $$
\end{theo}

Under Assumption \ref{assmuK1}, this result has been proven in \cite{champagnat2006microscopic}.\\

We are interested in the effect of selective sweep by recurrent mutations on the genetic diversity. 
This diversity will be expressed in terms of the distribution of mutant haplotypes $(N_1^K(t),N_2^K(t),... )$ at the end of the sweep.
In other words, we want to know if $a$-individuals sampled at the end of the sweep originate from the same 
or from distinct $a$-mutant individuals.
In order to describe the haplotype distribution, we introduce the so called GEM distribution, named after 
Griffiths, Egen and McCloskey:

\begin{defi} \label{defiGEM}
Let $\rho$ be positive.
The infinite sequence of random variables $(P_i,i \geq 1)$ has a GEM distribution with parameter $\rho>0$
if for every $i \geq 1$,
$$ P_i \overset{\mathcal{L}}{=}B_i \prod_{j=1}^{i-1}(1-B_j),  $$
where $(B_i,i \geq 1)$ is a sequence of i.i.d. random variables with law $\textrm{Beta}(1,\rho)$ whose density with 
respect to the Lebesgue measure is 
$$ \rho(1-x)^{\rho-1}\mathbf{1}_{[0,1]}(x)  .$$
\end{defi}

We are now able to describe the relative proportions of the different haplotypes at the end 
of the sweep (plus a time negligible with respect to $K$):

\begin{theo} \label{theoprecise}
Assume that \eqref{condfitnesses} and \eqref{cond_ini} hold.
 For $i \geq 1$ and $t\geq 0$, denote by $N_a^{(i),K}(t)$ the size at time $t$ of the $i$th oldest family among the mutant populations having survived until time $t$. 
Let $f_K$ be a positive function of $K$ such that $f_K= o(K)$, $G$ a non-negative constant, and
$$ \mathcal{T}_K:= T_F^K + Gf_K. $$
 Then,
 \begin{enumerate}
  \item Under Assumption \ref{assmuK1} (see \cite{champagnat2006microscopic}), 
$$ (N_a^K(\mathcal{T}_K))^{-1}(N_a^{(1),K}(\mathcal{T}_K),N_a^{(2),K}(\mathcal{T}_K),...) \underset{K \to \infty}{\to} (1,0,0,...),
\quad \text{in probability}, $$
 \item Under Assumption \ref{assmuK2}, 
$$ (N_a^K(\mathcal{T}_K))^{-1}(N_a^{(1),K}(\mathcal{T}_K),N_a^{(2),K}(\mathcal{T}_K),...) \underset{K \to \infty}{\to} (P_1,P_2,P_3...), 
\quad \text{in law}, $$
where $(P_i,i \geq 1)$ has a GEM distribution with parameter $f_A\bar{n}_A\lambda^{aA}/f_a$.
 \item Under Assumption \ref{assmuK3}, 
$$ (N_a^K(\mathcal{T}_K))^{-1}(N_a^{(1),K}(\mathcal{T}_K),N_a^{(2),K}(\mathcal{T}_K),...) \underset{K \to \infty}{\to} (0,0,0...), 
\quad \text{in probability}, $$
\end{enumerate}
where for $x \in \R^\N$, $ \|x\|= \sup_{i \in \N}|x_i|. $
\end{theo}

As a consequence, the relative haplotype proportions at the end of the sweep essentially differ depending on the 
mutation probability scale. 

 Theorems \ref{theoduration} and \ref{theoprecise} are still valid if, instead of considering mutation events we consider migration events occurring at a constant rate. In this case, 
 if we denote by $\mu_K^a$ the migration rates and see $\mu_K^{aA}$ as the probability that a new born emigrates, we have to impose the conditions 
$\mu_K^a \ll \lambda^a/\log K$  (Assumption \ref{assmuK1}), 
$\mu_K^a \sim \lambda^a >0$  (Assumption \ref{assmuK2}), $\mu_K^a \sim \lambda^a K^\gamma$  (Assumption \ref{assmuK3}).

\begin{rem}
 Theorem \ref{theoprecise} point (2) shows that in our case taking into account the changes in population size does not modify the type of the haplotype distribution.
A GEM distribution with parameter $f_A\bar{n}_A\lambda^{aA}/f_a$ is equivalent to an Ewens sampling formula with parameter $\Theta:=2f_A\bar{n}_A\lambda^{aA}/f_a$.
 Such a pattern has already been found when the population is modeled as a Wright-Fisher process \cite{pennings2006soft}, but with a parameter 
 (with our notations) $\Theta:=2\bar{n}_A\lambda^{aA}$. We will comment this difference in Section \ref{Discussion}.
\end{rem}

Theorem \ref{theoprecise} allows us to answer the following question: what is the probability that two individuals sampled 
uniformly at random at time $\mathcal{T}_K$ are identical by descent, or in other words originate from 
the same individual alive at time $0$? 
To answer this question, we need to define an equivalence relation between two sampled individuals $i$ and $j$: 
\begin{eqnarray*}\label{rel_equivalence}
 i \sim j & \Longleftrightarrow & \Big\{\text{there is no mutation in the lineages of $i$ and $j$ sampled at time $\mathcal{T}_K$} \nonumber \\
 && \text{before their more recent common ancestor (backward in time)}\Big\}.
\end{eqnarray*}
We do not indicate the dependence on $\mathcal{T}_K$ for the sake of readability but it will be clear in the statement of the 
results. Note that the probability of the event $i \sim j$ does not depend on the labeling of $a$-individuals at time 
$\mathcal{T}_K$ as they are exchangeable. The following result is a direct consequence of Theorem \ref{theoprecise}:

\begin{cor}\label{coro_id_by_descent} 
Assume that \eqref{condfitnesses} and \eqref{cond_ini} hold.
Then,
 \begin{enumerate}
  \item Under Assumption \ref{assmuK1} (see \cite{champagnat2006microscopic}), 
  $$ \lim_{K \to \infty} \P(i \sim j , \forall \text{ individuals $i$ and $j$ alive at time } \mathcal{T}_K)=1. $$
  \item Under Assumption \ref{assmuK2}, 
\begin{multline*}
\lim_{K \to \infty} \P(i \sim j , \forall \text{ $a$-individuals $i$ and $j$ sampled }\\
\text{uniformly at random at time } \mathcal{T}_K)=\frac{1}{1+ 2\bar{n}_Af_A\lambda^{Aa}/f_a}. 
\end{multline*}
  \item Under Assumption \ref{assmuK3}, 
\begin{multline*}
\lim_{K \to \infty} \P(i \sim j , \forall \text{ $a$-individuals $i$ and $j$ sampled uniformly at random at time } \mathcal{T}_K)=0. 
\end{multline*}
\end{enumerate}
\end{cor}

We end this section by the limit behaviour of the $A$- and $a$-population size dynamics 
when the mutation probabilities per reproductive event are of order one. 
We do not assume anymore Condition \eqref{cond_ini} on population state at time $0$.
In this case we get a deterministic limit. 
More precisely, a direct application of Theorem 2.1 p. 456 in Ethier and Kurtz \cite{ethiermarkov} gives 

\begin{lem}
  Let $T$ be a positive real number. Assume that
 $C_{\alpha, \alpha'}>0$ and $f_\alpha>D_\alpha$ for $(\alpha , \alpha') \in \mathcal{A}^2$ and that
 $(N_A^K(0)/K,N_a^K(0)/K)$ converges in probability 
 to $z(0)=(z_A(0),z_a(0)) \in (\R_+^*)^A \times (\R_+^*)^a$ when $K \to \infty$.
 Then under Assumption \ref{assmuK4} and when $K \to \infty$, the process 
 $$(N_A^K(t)/K,N_a^K(t)/K, 0\leq t \leq T)$$ 
 converges in probability on $[0,T]$ for the
uniform norm to the deterministic solution of 
\begin{equation} \label{system_high_mut} \left\{\begin{array}{ll}
 \dot{n}_A= (f_A(1- \lambda^{Aa})-D_A-C_{A,A}n_A-C_{A,a}n_a)n_A+ f_a\lambda^{aA}n_a, &n_A(0)= z_A(0)\\
 \dot{n}_a= (f_a(1- \lambda^{aA})-D_a-C_{a,A}n_A-C_{a,a}n_a)n_a+ f_A\lambda^{Aa}n_A, &n_a(0)= z_a(0).
\end{array}\right. 
\end{equation}
\end{lem}

To specify the properties of this dynamical system we need to introduce some notations:
\begin{equation}\label{defrhoalpha}
 \rho_\alpha:= f_\alpha (1- \lambda^{\alpha \bar{\alpha}})- D_\alpha, \quad \alpha \in \mathcal{A},
\end{equation}
\begin{equation}\label{defp}
 p:= \frac{\rho_A C_{a,A}- \rho_a C_{A,A}-f_A\lambda^{Aa}C_{Aa}}{f_a \lambda^{aA}C_{a,a}} 
 - \frac{(\rho_A C_{a,a}- \rho_a C_{A,a}+f_a\lambda^{aA}C_{aA})^2}{3(f_a \lambda^{aA}C_{a,a})^2},
\end{equation}
\begin{multline} \label{defq}
 q:= \frac{\rho_A C_{a,a}- \rho_a C_{A,a}+f_a\lambda^{aA}C_{aA}}{27f_a \lambda^{aA}C_{a,a}} 
 \left( \frac{2(\rho_A C_{a,a}- \rho_a C_{A,a}+f_a\lambda^{aA}C_{aA})^2}{(f_a \lambda^{aA}C_{a,a})^2}-\right.\\
\left.\frac{9\rho_A C_{a,A}- \rho_a C_{A,A}-f_A\lambda^{Aa}C_{Aa}}{f_a \lambda^{aA}C_{a,a}}  \right)
-\frac{f_A \lambda^{Aa}C_{A,A}}{f_a \lambda^{aA}C_{a,a}},
\end{multline} 
\begin{equation}\label{defr}
 r:= \frac{\rho_AC_{a,a}- \rho_aC_{A,a}+f_a\lambda^{aA}C_{aA}}{3f_a \lambda^{aA}C_{a,a}} \quad \text{and} \quad \Delta:= -(4p^3+27q^2).
\end{equation}
The term $\rho_\alpha$ corresponds to the growth rate of a small $\alpha$-population (when we can neglect competition). The other parameters are technical 
terms needed to describe the different types of fixed points that we will encouter (see Figure \ref{typesptsfixes}).
\begin{figure}[h]
    \centering
     \includegraphics[width=13cm,height=3cm]{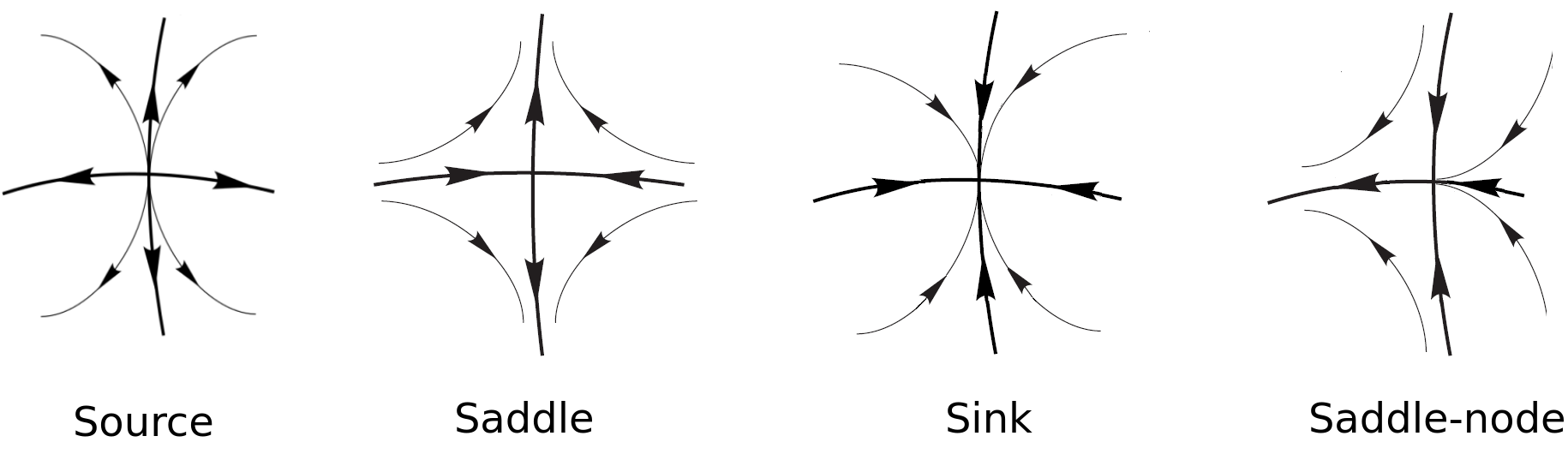}
     \caption{The different types of fixed points of system \eqref{system_high_mut}
     (adapted from Figure 2.13 in \cite{dumortier2006qualitative}).}
   \label{typesptsfixes}
\end{figure}

Then we have the following result when the mutation probabilities do not vanish for large $K$:

\begin{theo} \label{theo_high_mut}
The dynamical system \eqref{system_high_mut} has no periodic orbit in $\R_+^2$ and admits at least two fixed points in $\R_+^2$: 
$(0,0)$, 
and at least one fixed point in $(\R_+^*)^2$.
Moreover, if $\rho_\alpha>0$, $\alpha \in \mathcal{A}$, $(0,0)$ is unstable and,
\begin{itemize}
 \item If $\Delta<0$, then there is only one fixed point in $(\R_+^*)^2$,
 $$ \tilde{n}_a= \rho \tilde{n}_A= \frac{\rho(\rho_A+f_a \lambda^{aA}\rho)}{C_{A,A}+C_{A,a}\rho}, $$
where
$$\rho= r+ \left( \frac{-q-\sqrt{|\Delta|/27}}{2} \right)^{1/3}+\left( \frac{-q+\sqrt{|\Delta|/27}}{2} \right)^{1/3} .$$
\item If $\Delta=0$, then there are two fixed points in $(\R_+^*)^2$,
 $$ \tilde{n}^{(i)}_a= \rho^{(i)} \tilde{n}^{(i)}_A= \frac{\rho^{(i)}(\rho_A+f_a \lambda^{aA}\rho^{(i)})}
 {C_{A,A}+C_{A,a}\rho^{(i)}}, \quad i \in \{1,2\}, $$
where
$$\rho^{(1)}= r+ \frac{3q}{p}\quad \text{and} \quad \rho^{(2)}= r -\frac{3q}{2p} .$$
\item If $\Delta>0$, we introduce the real numbers
 $$ \rho^{(i)}:=r+ 2 \sqrt{\frac{-p}{3}}\cos \left( \frac{1}{3} \arccos \left( \frac{-q}{2}\sqrt{\frac{27}{-p^3}} \right) + \frac{2i\pi}{3} \right), \quad  i \in \{0,1,2\}.$$
 Then the fixed points in $(\R_+^*)^2$ are,
 $$ \tilde{n}_a= \rho^{(i)} \tilde{n}_A= \frac{\rho^{(i)}(f_A(1- \lambda^{Aa})-D_A+f_a \lambda^{aA}\rho^{(i)})}
 {C_{A,A}+C_{A,a}\rho^{(i)}}, \  \text{for} \  i \in \{ 0,1,2 \} \  \text{such that} \  \rho^{(i)}>0. $$
As a consequence, there is one, two, or three fixed points in $(\R_+^*)^2$ depending on the signs of $(\rho^{(i)}, i \in \{0,1,2\})$.
\end{itemize}
When there is one fixed point in $(\R_+^*)^2$, it is a sink, when there are two fixed points, one of them is a sink, 
and the second one a saddle-node, and when there are three fixed points, either there are two sinks and one saddle, or there 
are two saddle-nodes and a sink.

Moreover, if $\rho_A\rho_a>f_a\lambda^{aA}f_A\lambda^{Aa}$, $(0,0)$ is a source, and if $\rho_A\rho_a<f_a\lambda^{aA}f_A\lambda^{Aa}$, 
$(0,0)$ is a saddle.
\end{theo}

The conditions implying the different patterns of fixed points in Theorem \ref{theo_high_mut} are not intuitive, 
and the case $\rho_\alpha<0$ for $ \alpha \in \mathcal{A}$ is not considered. 
We now introduce more intuitive sufficient conditions to have a unique stable fixed point for the system \eqref{system_high_mut}, 
taking into account the case $\rho_\alpha<0$ for $ \alpha \in \mathcal{A}$.

\begin{pro}\label{propintuit}
The system \eqref{system_high_mut} satifies the following properties:
\begin{itemize}
\item If $C_{a,a}C_{A,A}>C_{A,a}C_{a,A}$, 
then there is only one fixed point in $(\R_+^*)^2$ and this fixed point is a sink.
\item If there exists $\alpha$ in $\mathcal{A}$ such that  
 $$
f_{\bar{\alpha}}\lambda^{\bar{\alpha}\alpha}>
\frac{\rho_\alpha}{C_{\alpha,\alpha}}C_{\alpha,\bar{\alpha}}, $$
then there is only one fixed point $\tilde{n}$ in $\R_+^* \times \R_+^*$ and this fixed point is a sink.
Moreover it satisfies
$$ \frac{ f_{\bar{\alpha}}\lambda^{\bar{\alpha}\alpha}}{C_{\alpha,\bar{\alpha}}} >\tilde{n}_\alpha >\frac{\rho_\alpha}{C_{\alpha,\alpha}} $$
and, for $* \in \{<,=,>\}$,
$$ f_{\alpha}\lambda^{\alpha\bar{\alpha}}* \frac{\rho_{\bar{\alpha}}}{C_{\bar{\alpha},\bar{\alpha}}}C_{\bar{\alpha},\alpha}
\Longleftrightarrow 
 \frac{ f_{\alpha}\lambda^{\alpha\bar{\alpha}}}{C_{\bar{\alpha},\alpha}} *\tilde{n}_{\bar{\alpha}} *\frac{\rho_{\bar{\alpha}}}{C_{\bar{\alpha},\bar{\alpha}}}.
$$\end{itemize} 
\end{pro}

 The term 
 $$ \Delta(\alpha, \bar{\alpha}):=f_{\bar{\alpha}}\lambda^{\bar{\alpha}\alpha}- \frac{\rho_\alpha}{C_{\alpha,\alpha}}C_{\alpha,\bar{\alpha}}$$
 appearing in Theorem \ref{theo_high_mut} has a biological interpretation: 
 $$n^*_\alpha:=\rho_\alpha/C_{\alpha,\alpha}$$ 
 is the equilibrium size of an $\alpha$-population in which a fraction $\lambda^{\alpha\bar{\alpha}}$ of offspring emigrates, and the term  
 $\Delta(\alpha, \bar{\alpha})$ 
 compares
 the mean number of $\alpha$-individuals killed by competition against a new $\bar{\alpha}$-mutant occurring in this $\alpha$-population 
 of size $n^*_\alpha$, 
 with the mean number of new $\alpha$-individuals which are created by mutations by this new mutant $\bar{\alpha}$.

\begin{rem}
In \cite{billiard2016}, Billiard and coauthors study the effect of horizontal transfer on the dynamics of a stochastic population with 
two allelic types. They also get a two dimensional
dynamical system as a limit for the rescaled process in large population, which is a competitive Lotka-Volterra system plus some terms 
representing the exchange of genetic material by 
means of horizontal transfer. 
Interestingly, even if their system is more complex than \eqref{system_high_mut} (horizontal transfer may be frequency- or 
density-dependent unlike mutations), they obtain 
similar behaviours. For example, they get up to three fixed points in the positive quadrant and the possibility of having two stable 
fixed points in the positive quadrant.
Moreover, the sign of the invasion fitnesses do not 
indicate anymore if an allelic type is 
"favorable" or not. 
We have a similar phenomenon in our case. When the mutation probabilities per reproductive event are non negligible with respect to $1$ 
(which is the case under Assumption \ref{assmuK4}) the invasion fitnesses read:
$$ \tilde{S}_{\alpha \bar{\alpha}}:=\rho_\alpha- \frac{C_{\alpha ,\bar{\alpha}}}{C_{\bar{\alpha}, \bar{\alpha}}}\rho_{\bar{\alpha}}.$$
But according to Proposition \ref{propintuit} we can have for example $\tilde{S}_{aA}, \tilde{S}_{Aa}<0$ but one stable polymorphic equilibrium 
(if $f_{\bar{\alpha}}\lambda^{\bar{\alpha}\alpha}/\rho_{\alpha}>C_{\alpha, \bar{\alpha}}/C_{\alpha,\alpha}$).
\end{rem}

We end this section with two convergence results for the dynamical system \eqref{system_high_mut} under some conditions 
on the parameters. 
Before stating these results, we recall the stable fixed points of the system \eqref{S} 
(which corresponds to the particular case of the system \eqref{system_high_mut} when the mutation probabilities are null):
\begin{itemize}
 \item If $S_{aA}>0>S_{Aa}$, then \eqref{S} has only one stable fixed point, which attracts all the trajectories with a positive initial density of $a$-individuals:
 \begin{equation}\label{nbarmono}
\bar{n}^{(a)}= \left( 0,\bar{n}_a \right),  
 \end{equation}
where $\bar{n}_a$ has been defined in \eqref{defbarn}.
 \item If $S_{Aa}>0, \ S_{aA}>0$, then \eqref{S} has only one stable fixed point, which attracts all the trajectories with a positive initial density of $A$- and $a$-individuals:
\begin{equation}\label{defeqpospos}
 \bar{n}^{(aA)}=\left( \frac{C_{a,a}(\beta_A-\delta_A)-C_{A,a}(\beta_a-\delta_a)}
{C_{A,A}C_{a,a}- C_{A,a}C_{a,A}},\frac{C_{A,A}(\beta_a-\delta_a)-C_{a,A}(\beta_A-\delta_A)}
{C_{a,a}C_{A,A}- C_{a,A}C_{A,a}} \right)                                                                               
\end{equation}
\end{itemize}

The first result of Proposition \ref{prop_high_mut2} has been stated in Theorem 1.2 \cite{coville2013convergence}, and the second one will be proven in Section \ref{proof_prop_conv}.

\begin{pro} \label{prop_high_mut2}
 \begin{itemize}
  \item Assume that $C_{A,\alpha}=C_{a,\alpha}$ for $\alpha \in \mathcal{A}$, and that $f_A\lambda^{Aa}= f_a\lambda^{aA}$.
Then for any nonnegative and nonnull initial condition, the solution of \eqref{system_high_mut} 
converges exponentially fast to its unique stable equilibrium.
\item Let us consider the following dynamical system for $\lambda, p \geq 0$: 
\begin{equation} \label{system_high_mut_bis} \left\{\begin{array}{ll}
 \dot{n}_A= (f_A(1- p\lambda)-D_A-C_{A,A}n_A-C_{A,a}n_a)n_A+ f_a\lambda n_a, &n_A(0)= z_A(0)\\
 \dot{n}_a= (f_a(1- \lambda)-D_a-C_{a,A}n_A-C_{a,a}n_a)n_a+ f_Ap \lambda n_A, &n_a(0)= z_a(0).
\end{array}\right. 
\end{equation}
Assume that $f_\alpha-D_\alpha>0, \alpha \in \mathcal{A}$, and that
$$S_{aA}>0>S_{Aa},\quad \text{or} \quad S_{Aa}>0, \ S_{aA}>0.$$
Then there exists $\lambda_0(p)$ such that for all $\lambda\leq \lambda_0(p)$ the dynamical system \eqref{system_high_mut_bis} converges to its 
unique equilibrium $\bar{n}^{(\lambda)}$. Moreover, if $S_{aA}>0>S_{Aa}$, 
$$ \bar{n}^{(\lambda)}_A= \frac{f_a (f_a-D_a)}{C_{A,a}S_{aA}}\lambda + O(\lambda^2)$$
and
$$\bar{n}^{(\lambda)}_a= \bar{n}_a - \frac{f_a}{C_{a,a}}\frac{C_{A,a}S_{aA} + (f_a-D_a)C_{a,A}}{C_{A,a}S_{aA}}\lambda + O(\lambda^2) ,$$
and if $S_{Aa}>0, \ S_{aA}>0$, 
\begin{multline*}
 (C_{a,a}C_{A,A}-C_{a,A}C_{A,a})(\bar{n}_A^{(\lambda)}-\bar{n}^{(aA)}_A)=\\
 \lambda  \left[ C_{a,a}\left(\frac{f_a\bar{n}^{(aA)}_a}{\bar{n}^{(aA)}_A}-f_Ap\right)
- C_{A,a}\left(\frac{f_A p \bar{n}^{(aA)}_A}{\bar{n}^{(aA)}_a}-f_a\right)\right]+ O(\lambda^2),
\end{multline*}
and
\begin{multline*}
 (C_{a,a}C_{A,A}-C_{a,A}C_{A,a})(\bar{n}_a^{(\lambda)}-\bar{n}^{(aA)}_a)=\\
 \lambda  \left[ C_{A,A}\left(\frac{f_A p \bar{n}^{(aA)}_A}{\bar{n}^{(aA)}_a}-f_a\right)
-C_{a,A}\left(\frac{f_a\bar{n}^{(aA)}_a}{\bar{n}^{(aA)}_A}-f_Ap\right)
\right]+ O(\lambda^2).
\end{multline*}
 \end{itemize}
\end{pro}

Notice that when the $\lambda^{\alpha \bar{\alpha}}$, $\alpha \in \mathcal{A}$, are close to $0$ (Proposition \ref{prop_high_mut2}) or 
$1$ ($\lambda^{\alpha \bar{\alpha}}> (f_\alpha- D_\alpha)/f_\alpha $ for an $\alpha 
\in \mathcal{A}$ is enough according to Proposition \ref{propintuit}), the behaviour of the solutions of 
the dynamical system \eqref{system_high_mut} is simple: 
a unique stable equilibrium in the positive quadrant, and $\textbf{0}$ as an unstable fixed point. 
More complex dynamics, such as the presence of two stable fixed points in the positive quadrant, appear 
for intermediate values of the mutation probabilities (Theorem \ref{theo_high_mut}).\\

In the sequel for the sake of readability we will write $N_.$ instead of $N_.^K$.
We will denote by $c$ a positive constant whose value can change from line to line, and $\eps$ will be 
a small positive number independent of $K$.

\section{Discussion} \label{Discussion}  
 
 A selective sweep was first defined as the reduction of genetic diversity around a positively selected allele. 
 But the pattern of genetic diversity around this allele can change significantly if this beneficial allele is descended from more
than a single individual at the beginning of the selective phase. 
Indeed, in this case, genetic variation that is linked to any ancestor of the
beneficial allele will survive the selective phase and the reduction
in diversity is less severe. Pennings and Hermisson \cite{pennings2006soft} called the resulting pattern \textit{soft sweep} 
to distinguish it from the classical \textit{hard sweep} from only a single origin. 
Various scenarios may lead to soft sweeps. For example, multiple copies of the beneficial
allele can already be present in the population before the selective phase (adaptation from
standing genetic variation  \cite{prezeworski2005signature,hermisson2005soft,smadi2014eco}). 
They can also appear recurrently as new mutations during the selective phase.
The results of Pennings and Hermisson \cite{pennings2006soft}, Hermisson and Pfaffelhuber \cite{hermisson2008pattern} as well as our Corollary 2.1 show
that the probability of soft selective sweeps is mainly dependent on the product of population size (which has the order of the carrying 
capacity $K$) and mutation probability per reproductive event, $\theta_K:=\bar{n}_A K \mu^{Aa}$. 
Hard sweeps are likely when $\theta_K$ is close to $0$. For larger $\theta_K$,
soft sweeps become more likely.
 
 \subsection{Varying population size}
The fact that the population size is not constant has no macroscopic influence under Assumptions $1$ to $3$ (except that proofs are more technical). 
Indeed the equilibrium between the two alleles is already attained at the end of the first phase of the sweep (see Section \ref{poisson_rep} 
for a precise definition) during which the total population size stays close to $\bar{n}_AK$. In particular, the only parameters which appear in 
Theorem \ref{theoprecise} to describe the distribution of haplotypes are the quotient of birth rates $f_A/f_a$, as well as $\theta_K$.
To see a significant effect of population size changes we should observe a change of order $K$ of the population size during the first phase, or 
consider an initial population size $N_0$ such that  $N_0/K$ or $K/N_0$ is very large.
For a study of the probability of soft sweeps when the population changes over time we refer to \cite{wilson2014soft}. 
In this paper the authors show that this probability is no more essentially determined by the parameter $\theta_K$ but also by the strength of the selection 
($S_{aA}/f_a$ with our notations). 

Under Assumption \ref{assmuK4} the variation of population size is important as the allele dynamics are given by a Lotka-Volterra type equation with additional 
mutations. 
This equation may have two stable coexisting equilibria with different population sizes.
The equilibria strongly depend on the values of the competition parameters.

 \subsection{Eco-evolutionary approach}
The aim of the eco-evolutionary approach is to take into account the reciprocal interactions between ecology and evolution. 
Each individual
is characterized by its birth rate, intrinsic death rate and competition
kernel with other types of individuals, and the differential 
reproductive success of individuals generated by their interactions entails progressive
variations in the number of individuals carrying a given genotype.
This way of modeling interacting individuals allows one to better understand the interplay between the different traits of individuals in the population dynamics 
than population genetic models were an individual is often only characterized by its ``fitness'', independently of the population state.
In the present work for instance we see that under Assumptions \ref{assmuK1} to \ref{assmuK3}, the parameter $f_A/f_a$ is important, whereas it was implicitely 
taken equal to $1$ in previous studies of soft sweeps from recurrent mutations. 
Similarly in the study of genetic hitchhiking during a hard sweep, results from population genetics models can be recovered if we assume 
$f_A=f_a$ and $S_{aA}/f_a=|S_{Aa}|/f_A$ \cite{schweinsberg2005random,brink2015stochastic,rebekka2015genealogies}.
And under Assumption \ref{assmuK4}, if we consider competitions independent of the individuals' type, as in population genetics models, 
Equation \eqref{system_high_mut} admits at most one equilibrium in the positive quadrant.

More importantly this approach allows one to account for non transitive relations between mutants: if a mutant $1$ invades a resident population $0$ 
and is invaded by an other mutant population $2$, this does not necessarily imply that the resident population $0$ would have been invaded by the 
mutant population $2$ \cite{billiard2015interplay,billiard2016}. 
Frequency-dependent selection is often ignored in population genetics model, which does not allow to model such a phenomenon

\subsection{Back mutations} \label{backmut}
Another novel aspect of this work is to take into account the possibility of back mutations. As for the population size variations we prove that 
their role on the haplotype diversity is negligible under Assumptions \ref{assmuK1} to \ref{assmuK3}. 
This is for the same reason: the most important phase is the first phase of the sweep during which the $a$-population size is small. As a consequence, it 
produces a small number of back mutations.
In contrast, back mutations may strongly
affect the population dynamics under Assumption \ref{assmuK4}. Indeed, the dynamical system \eqref{system_high_mut} may admit up to two stable fixed 
points and up to three fixed points in the positive quadrant $(\R_+^*)^2$, whereas a two-dimensional competitive Lotka-Volterra system has at most one fixed point 
in the positive quadrant.

 \section{Poisson representation and couplings}\label{sectionpoisson}

 In this section, we construct the population process by means of Poisson point processes, and we describe two couplings with birth 
 and death processes without competition 
 which will be needed in the proofs.
 
 \subsection{Poisson representation} \label{poisson_rep}
 
Following the definition of the process in \eqref{birth_rate} and \eqref{deathrate} 
we can associate a birth rate and a death rate to each subpopulation $N_i$ originating from 
the $i$th first generation mutant of type $a$: 
\begin{equation}\label{birthrateNi}
 b_i^K(n)= (f_a- \mu_K^{aA}) n_i, 
\end{equation}
 and
\begin{equation}\label{deathrateNi}
  d_i^K (n)= \Big(D_a + \frac{C_{a,A}}{K}n_A + \frac{C_{a,a}}{K}n_a\Big)n_i,
\end{equation}
where $n=(n_a,n_1,n_2,...)$ is the current state of the population, and $n_a = \sum_{i \in \N}n_i$.

In the vein of Fournier and M\'el\'eard \cite{fournier2004microscopic} we represent the $A$- and  $a$-population sizes in terms 
of Poisson measures. Let $(Q_A(ds,d\theta),Q_m(ds,d\theta),Q_i(ds, d\theta), i \in \N)$ 
be independent Poisson random measures on $\R_+^2$ with intensity $ds d\theta$, and $(e_A, e_i, i \in \N )$ be the canonical 
basis of $\R^{A \times \N} $.
We decompose on possible jumps that may occur: births and deaths of $A$-individuals, arrivals of new mutants of type $a$, 
and births and deaths of $a$-individuals 
without mutation (we do not distinguish mutant and non mutant $A$ individuals, as it will not be needed in the proofs).
Let $N_m$ denote the point process describing the arrival of $a$ mutants:
\begin{equation}\label{deNm}
 N_m(t):= \int_0^t \int_{\R_+}\mathbf{1}_{\theta\leq f_A\mu_K^{Aa}N_A({s^-})}Q_m(ds,d\theta).
\end{equation}
Remark that with this notation, 
$$T_i= \inf \{t \geq 0, N_m(t)=i\},$$
with the convention $\inf \emptyset = \infty$.
Recall the birth and death rates of the process in \eqref{birth_rate}, \eqref{deathrate}, \eqref{birthrateNi} and \eqref{deathrateNi}.
Then for 
every real function $h$ on $\R_+^{A \times \N}$ measurable and such that $h(N(t))$ is integrable:
\begin{multline}
 \label{defN}
 h(N(t))=h(N(0))+ \int_0^t\int_{\R_+}  \Big( \Big( 
h(N({s^-})+e_A)-h(N({s^-}))\Big)\mathbf{1}_{\theta\leq b^K_A(N({s^-}))}\\
+\Big(h(N({s^-})-e_A)-h(N({s^-}))\Big)\mathbf{1}_{0<\theta-b^K_{A}(N({s^-}))\leq d_{A}^K(N({s^-}))} \Big)Q_A(ds,d\theta)\\
+  \int_0^t\int_{\R_+}  \Big(h(N({s^-})+e_{N_m(t)})-h(N({s^-}))\Big) \mathbf{1}_{\theta\leq f_A\mu_K^{Aa}N_A({s^-})}Q_m(ds,d\theta) \\
+\underset{i \in \N }{ \sum } \Big(  \int_{0}^t\int_{\R_+}  \Big( 
\Big(h(N({s^-})+e_i)-h(N({s^-}))\Big)\mathbf{1}_{\theta\leq b^K_i(N({s^-}))}\\
+\Big(h(N({s^-})-e_i)-h(N({s^-}))\Big)\mathbf{1}_{0<\theta-b^K_{i}(N({s^-}))\leq d_{i}^K(N({s^-}))} \Big)Q_i(ds,d\theta) \Big).
\end{multline}
In particular, taking $h(N)=\sum_{i \in \N} N_i$ in \eqref{defN} we get:
\begin{equation*}
 N_a(t)= \underset{i \in \N }{ \sum } \mathbf{1}_{t \geq T_i}\Big\{1+  \int_{0}^t\int_{\R_+}  \Big( 
\mathbf{1}_{\theta\leq f_a N_i({s^-})}-\mathbf{1}_{0<\theta-f_aN_i({s^-})\leq d_{i}^K(N({s^-}))}\Big)Q_i(ds,d\theta) \Big\}.
\end{equation*}

Let us now introduce a finite subset of $\N$ containing 
the equilibrium size of the $A$-population,
\begin{equation} \label{compact1}I_\eps^K:= \Big[K\Big(\bar{n}_A-2\eps \frac{C_{A,a}}{C_{A,A}}\Big),K\Big(\bar{n}_A+2\eps \frac{C_{A,a}}{C_{A,A}}\Big)\Big]\cap \N, \end{equation}
and the stopping times $T^K_\eps$ and $S^K_\eps$, which denote respectively the hitting time of $\lfloor\eps K \rfloor$ by the mutant 
a-population and the exit time of $I_\eps^K$ by the resident $A$-population,
\begin{equation} \label{TKTKeps1} T^K_\eps := \inf \Big\{ t \geq 0, N^K_a(t)= \lfloor \eps K \rfloor \Big\},\quad S^K_\eps := \inf \Big\{ t \geq 0, N^K_A(t)\notin I_\eps^K \Big\}.  \end{equation}

We call "first phase" of the selective sweep the time interval needed for the mutant $a$-population to hit a size $\lfloor \eps K \rfloor$. 
Let us first recall what happens when there is only one mutant $a$ occurring in an $A$-population at its equilibrium size $\lfloor \bar{n}_A K \rfloor$.
It is stated in \cite{champagnat2006microscopic} that the $a$-population generated by this mutant has a probability 
close to 
${S_{aA}}/{f_a}$
to hit a size $ \lfloor \eps K \rfloor$, and that there exist two positive 
constants $c$ and $\eps_0$ such that for every $\eps \leq \eps_0$, 
\begin{equation*}\label{compSepsTeps}
 \limsup_{K \to \infty} \P(T_\eps^K<\infty, T_\eps^K > S_\eps^K)\leq c \eps,
\end{equation*}
and
$$ \Big|\P\Big((1-c\eps)\frac{\log K}{S_{aA}} \leq T_\eps^K \leq  (1+c\eps)\frac{\log K}{S_{aA}} \Big| T_\eps^K<\infty \Big)
- \frac{S_{aA}}{f_a}\Big| \leq  c\eps.  $$
We will see that 
in the case of recurrent mutations the $a$-population hits the size $\lfloor \eps K \rfloor$ with a probability close to $1$ and that
the duration of the first phase depends weakly on the value of $\lambda^{Aa}$ and $\lambda^{aA}$ under Assumptions 
\ref{assmuK1} to \ref{assmuK3}.
The proof of this result lies on couplings of the subpopulations $(N_i, i \geq 1)$ with independent birth and death processes that we will now describe.

\subsection{A first coupling with birth and death processes}

In this subsection, we introduce couplings of the population process with birth and death processes without competition, 
which will be needed to prove Theorem \ref{theoduration}.

We recall \eqref{TKTKeps1} and define the rescaled invasion fitness
\begin{equation}\label{defs}
 s=S_{aA}/f_a,
\end{equation}
and, for $\eps < S_{aA}/( 2 {C_{a,A}C_{A,a}}/{C_{A,A}}+C_{a,a} )$,
\begin{equation}\label{def_s_-s_+1}
s_-(\eps):=s-\eps\frac{ 2 {C_{a,A}C_{A,a}}+C_{a,a}{C_{A,A}}}{f_a{C_{A,A}}} \quad \text{and} \quad s_+(\eps):=s
+ 2\varepsilon \frac{C_{a,A}C_{A,a}}{f_aC_{A,A}}.
\end{equation}
Definitions \eqref{deffitinv} and 
\eqref{deathrateNi} ensure that
\begin{equation*}\label{ineqtxmort}
f_a (1-s_+(\eps))\leq \frac{{d}_{i}^K(N(t))}{N_i(t)}= f_a-S_{aA}+\frac{C_{a,A}}{K}(N_A(t)-\bar{n}_AK)+\frac{C_{a,a}}{K}N_a(t)\leq f_a (1-s_-(\eps)).
\end{equation*}
Let us define for $i \in \N$ two supercritical birth and death processes:
\begin{multline}\label{defNi*}
 N_i^{*}(t)= \mathbf{1}_{\{t \geq T_i\}}\Big(1+ 
 \int_{0}^{t}\int_{\R_+}  \Big( 
\mathbf{1}_{\theta\leq (1-\mu_K^{aA})f_a N_i^{*}({s^-})}\\
-\mathbf{1}_{0<\theta-(1-\mu_K^{aA})f_aN_i^{*}({s^-})\leq f_a(1-s_*(\eps))N_i^{*}({s^-})}\Big)Q_i(ds,d\theta) \Big),
\end{multline}
where $* \in \{-,+\}$.
These processes have two fundamental properties. First, almost surely:
\begin{equation}\label{couplage11}
 N_i^{-}(t) \leq N_i(t) \leq N_i^{+}(t), \quad \text{for all } t <  T^K_\eps\wedge S^K_\eps .
\end{equation}
Second 
$$\left((N_i^{-},N_i^{+}),i \in \N\right)$$
is a sequence of pairs independent conditionally on the mutation times $(T_i,i \in \N)$.
In other words during the first phase we can couple the subpopulations $(N_i, i \geq 1)$ with birth and death processes which evolve independently.
This will allow us to study the duration of the first phase.

\subsection{A second coupling with birth and death processes} \label{sectioncoupl2}


We will apply some results of Richard \cite{richard2011limit} on birth and death processes with immigration 
to prove Theorem \ref{theoprecise}.
In order to do this we need to introduce new couplings.
Let $N_m^{*}, * \in \{-,+,0\} $ be three counting processes, defined with the same 
Poisson measure $Q_m$ as $N_m$ (see Equation \eqref{deNm}) by
\begin{equation*}\label{deNmstar}
 N_m^{*}(t):= \int_0^t \int_{\R_+}\mathbf{1}_{\theta\leq \lambda_K^{*}(\eps)}Q_m(ds,d\theta),
\end{equation*}
where 
\begin{equation}\label{deflambda*}
\lambda_K^{0}(\eps):= f_A\mu_K^{Aa}K\bar{n}_A \quad \text{and}
\quad \lambda_K^{*}(\eps):= f_A\mu_K^{Aa}K(\bar{n}_A*2\eps {C_{A,a}}/{C_{A,A}}), \ * \in \{-,+\}. 
\end{equation}
Denote by $T_i^{*}$ their respective jump times
\begin{equation} \label{defTi*} T_i^{*}= \inf \{t \geq 0, N_m^{*}(t)=i\}.\end{equation}
The jump times $T_i^{*}$ are the arrival times of individuals generating new populations $Z_i^{(*)}$, which can be represented as follows: 
\begin{multline}\label{defZi*}
 Z_i^{*}(t)= \mathbf{1}_{\{t \geq T_i^{*}\}}\Big(1+ 
 \int_{0}^{t}\int_{\R_+}  \Big( 
\mathbf{1}_{\theta\leq f_a(1-\mu_K^{aA}) Z_i^{*}({s^-})}\\
-\mathbf{1}_{0<\theta-f_a(1-\mu_K^{aA})Z_i^{*}({s^-})\leq f_a(1-s_*(\eps))Z_i^{*}({s^-})}\Big)Q_i^{*}(ds,d\theta) \Big),
\end{multline}
where $s_0(\eps)=s$ has been defined in \eqref{defs}, and for $* \in \{-,+,0,\emptyset\}$ the Poisson measures $(Q_{j}^{*}, j \geq 1)$ are independent and chosen in a way such that:
$$ \text{If } ((i_1,*_1),  (i_2,*_2)) \in (\N \times\{-,+,0,\emptyset\})^2 \text{ such that } T_{i_1}^{*_1}=T_{i_2}^{*_2},
\text{then } Q_{i_1}^{*_1}=Q_{i_2}^{*_2} $$
(recall that the sequence of measures $(Q_i, i \geq 1)$ appeared in the definition of the population process in  \eqref{defN}).
Finally we introduce the sum of these population processes for $* \in \{-,+,0\}$
\begin{eqnarray}
 \label{defNi}
 Z^{*}(t)&=&\underset{i \in \N }{ \sum } Z_i^{*}(t) .
\end{eqnarray}

Then by construction we have almost surely:
\begin{equation*}\label{couplageT}
 T_i^{+}\wedge T^K_\eps\wedge S^K_\eps \leq T_i \wedge T^K_\eps\wedge S^K_\eps \leq T_i^{-}\wedge T^K_\eps\wedge S^K_\eps, \quad \text{for all } i \geq 1 .
\end{equation*}
\begin{equation}\label{couplageNi}
 Z_i^{-}(t) \leq N_i(t) \leq Z_i^{+}(t), \quad \text{for all $i$ such that } T_i^{+}=T_i=T_i^{-} \text{ and } t <  T^K_\eps\wedge S^K_\eps .
\end{equation}
\begin{equation}\label{couplageZ0}
 Z_i^{-}(t) \leq Z^{0}_i(t) \leq Z_i^{+}(t), \quad \text{for all $i$ such that } T_i^{+}=T_i^0=T_i^{-} \text{ and } t \geq 0 .
\end{equation}
and
\begin{equation*}\label{couplageZ}
 Z^{-}(t) \leq N_a(t) \leq Z^{+}(t), \quad \text{for all }t <  T^K_\eps\wedge S^K_\eps .
\end{equation*}

Notice that if the introductions of $Z^{-}$ and $Z^{+}$ are quite natural to get lower and upper bounds for the $a$ population 
size, the process $Z^{0}$ has been introduced only for technical reasons: we will need in Section \ref{sectionprecise} 
to apply limit theorems to 
population processes at time $|\ln \eps|/4S_{aA}$ when $\eps$ goes to $0$. 
This is feasible only if the definitions of the population processes 
do not depend on $\eps$, which is not the case of $Z^{-}$ and $Z^{+}$.

\section{Proof of Theorem \ref{theoduration}} \label{sectionproofduration}

Each mutant has a positive probability to generate a population which survives during a long time. As a consequence, the limit probability for the total 
$a$-population size to hit the value $\lfloor \eps K \rfloor$ when $K$ is large is equal to one. 
We do not prove the first point of Theorem \ref{theoduration} as it has already been stated in \cite{champagnat2006microscopic}.
Recall Definition \eqref{TKTKeps1}.
To prove the second point, we will first establish the following Lemma:

\begin{lem} \label{duration}
Under Assumption \ref{assmuK2} (in this case $\beta=0$) or \ref{assmuK3}, 
there exist two positive 
constants $c$ and $\eps_0$ such that for every $\eps \leq \eps_0$,
\begin{equation*}
  \liminf_{K \to \infty} \P\Big((1-c\eps)(1-\beta)\frac{\log K}{S_{aA}} \leq T_\eps^K \leq  (1+c\eps)(1-\beta)\frac{\log K}{S_{aA}} \Big)\geq 1- c\eps.  
\end{equation*} 
\end{lem}

\begin{proof}
The proof of Lemma \ref{duration} relies on comparisons of the $a$-population dynamics with these of branching processes. 
Broadly speaking mutants occur according to a Poisson 
process with parameter $f_A\bar{n}_A\lambda^{Aa}K^\beta$ and are successful (i.e. generate a large population) with probability $s$. 
Hence successful mutants follow approximately a Poisson process with parameter $f_A\bar{n}_As\lambda^{Aa}K^\beta$.
This is only an approximation, as the $A$-population size varies over time.
Once a number of order $K^\beta$ of successful mutant populations have 
occurred, which takes a time of order $1$, their total growth is close to this of a supercritical branching process with birth and death rates $f_a(1-\mu_K^{aA})$ and $f_a(1-s)$ 
conditioned on non-extinction and they take a time of order $(1-\beta)\log K/S_{aA}$ to hit a size $\lfloor\eps K\rfloor$ (see \eqref{equi_hitting}).\\

Before deriving upper and lower bounds for the time $T_\eps^K$, we recall that according to \eqref{temps_expo}
there exists a positive constant $V$ such that 
  \begin{equation}\label{compST}
   \lim_{K \to \infty}\P(S_\eps^K < T_\eps^K \wedge e^{KV}  )=0.
  \end{equation}
This is a direct consequence of Theorem 3 (c) in \cite{champagnat2006microscopic}, which is recalled in Appendix (see Lemma \ref{Th3cChamp}).\\

\noindent \textbf{Upper bound:}
Let us take $\eps_0$ such that $\bar{n}_A-2\eps_0 C_{A,a}/C_{A,A}\geq \bar{n}_A/2$.
By definition, on the time interval $[0, T_\eps^K \wedge S_\eps^K]$, new $a$-mutants occur with a rate larger than 
$$ f_AK\Big(\bar{n}_A-2\eps \frac{C_{A,a}}{C_{A,A}}\Big)\mu_K^{Aa}\geq f_A \bar{n}_A \frac{\lambda^{Aa}}{2}K^\beta $$
for $K$ large enough under Assumption \ref{assmuK2} or \ref{assmuK3}.
We couple each new mutant population $i$ with a birth and death process $N_i^{(-)}$ as described in \eqref{defNi*} and 
\eqref{couplage11}.
We introduce for every integer $k$ the first time $T_k(success)$ when $k$ "successful mutants" have occurred in the following sense:
$$ T_k(success):=\inf \Big\{ t \geq 0, \# \{ i, T_i \leq t, N_i^{-} \text{ does not get extinct} \}=k \Big\}. $$
Note that $T_k(success)$ is not a stopping time.
According to \eqref{hitting_times1} for each $i$ the probability that the process $N_i^{-}$ does not get extinct is
$s_-(\eps)\geq {s}/{2} $
for $\eps$ small enough.
Hence by construction, if $T^K_\eps\wedge S^K_\eps$ is large enough,  $T_k(success)$ is smaller than the sum of $k$ independent
exponentially distributed random variables with parameter 
$ f_A\bar{n}_A\lambda^{Aa}sK^\beta/4 $. As a consequence, adding \eqref{compST} we get that 
\begin{equation}\label{arriveepremimmortel}
 \liminf_{K \to \infty}\P(T_k(success)\wedge T_\eps^K\leq \eps \log K)\geq 1-c\eps
\end{equation}
for $\eps$ small enough, where $c$ is a positive constant.

The time needed by these $k$ successful mutant populations to hit a size $\lfloor \eps K \rfloor$ satisfies
\begin{equation}\label{majTi0epsK}
\lim_{K \to \infty} \P \Big( \frac{\inf \{ t-T_k(success) \geq 0, \sum_{1\leq m\leq k}N_{i_m}^{-}(t)=
\lfloor \eps K \rfloor\}}{(1-\beta)\log K}\leq \frac{1}{f_a s_-(\eps)} \Big)=1,
\end{equation}
where $N_{i_m}^{(-)}$ is the $m$th mutant population such that $N_{i_m}^{(-)}$ does not get extinct and we have applied \eqref{equi_hitting}.
Combining \eqref{compST}, \eqref{arriveepremimmortel} and \eqref{majTi0epsK} yields the existence of a positive $c$ such that for $\eps$ small enough
\begin{equation}\label{minduree}
 \liminf_{K \to \infty} \P\Big( T_{\eps}^K\leq  (1+c\eps)(1-\beta)\frac{\log K}{S_{aA}}\wedge S_\eps^K \Big)\geq 1- c\eps.
\end{equation}

\noindent \textbf{Lower bound:}
Equation \eqref{minduree} implies that with a probability close to one up to a constant times $\eps$, the number of $a$-mutants 
occurring during the first phase, $N_m(T_\eps^K)$, is smaller than a Poisson process with parameter 
$$ f_AK\Big(\bar{n}_A+2\eps \frac{C_{A,a}}{C_{A,A}}\Big)\mu_K^{Aa} $$
taken at time $(1+c\eps)(1-\beta){\log K}/{S_{aA}}$. Adding Assumption \ref{assmuK2} or \ref{assmuK3} yields
$$ \P(N_m(T_\eps^K)>K^{\beta+\eps})\leq  \frac{\E[N_m(T_\eps^K)]}{K^{\beta+\eps}} 
\leq  \frac{cK^\beta \log K}{K^{\beta+\eps}}\leq c \eps $$
for a constant $c$, $\eps$ small enough and $K$ large enough, where we used Markov Inequality. 
Using this last inequality and the definition of Coupling \eqref{couplage11}, we see 
that with a probability close to one the $a$-population size during the first phase is smaller than a birth and death process 
$N^{(\beta+\eps)}$ with individual birth rate $(1-\mu_K^{aA})f_a$, individual death rate $f_a(1-s_+(\eps))$ and initial state $K^{\beta+\eps}$.
But according to \eqref{equi_hitting}, there exists a positive $c$ such that for $\eps$ small enough
$$ \lim_{K \to \infty}\P \Big(N^{(\beta+\eps)}(t)< \lfloor \eps K \rfloor ,\ \forall t \leq (1-c\eps)(1-\beta)\frac{\log \lfloor \eps K \rfloor }{S_{aA}} \Big)=1. $$
This yields the desired lower bound for $T_\eps^K$ and ends the proof of Lemma \ref{duration}.
\end{proof}

Once the number of $a$-individuals is large enough, we can compare the evolution of the $a$- and $A$-population sizes with the solution of a deterministic 
system.
Recall that $n^{(z)}=(n^{(z)}_{\alpha},\alpha\in \mathcal{A})$ is the solution of the dynamical system \eqref{S}  
with initial condition $z \in \R_+^\mathcal{A}$. Then we have the following comparison result:

\begin{lem}\label{lemapprox}
 Let $z$ be in $(\R_+^*)^A \times (\R_+^*)^a $ and $T>0$. Suppose that Assumption \ref{assmuK2} or \ref{assmuK3} holds and that $N^K(0)/K$ converges to $z$ in probability. Then
\begin{equation*}\label{EK2}
\underset{K \to \infty}{\lim}\  \sup_{s\leq T}\ \|{N}^{K}(s)/K-n^{(z)}(s)  \|=0 \quad \text{in probability}
\end{equation*}
where $\|. \|$ denotes the $L^1$-Norm on $\R^\mathcal{A}$.
\end{lem}

The proof relies on a slight modification of Theorem 2.1 p. 456 in Ethier and Kurtz \cite{ethiermarkov}.
We do not detail it and refer the reader to Lemma 4.1 in \cite{smadi2014eco} for a similar derivation. 

Combining Lemmas \ref{duration} and \ref{lemapprox} we are now able to prove the point 2 of Theorem \ref{theoduration}.

\begin{proof}[Proof of Theorem \ref{theoduration}]
For $\eps\leq  C_{a,a}/C_{a,A} \wedge 2|S_{Aa}|/C_{A,a}$ and $z$ 
in $\R_+^{A}\times (\R_+^{a} \setminus \{0\} )$ we introduce a deterministic time $t_{\eps}(z)$ after 
which the solution $n^{(z)}$ of the dynamical system \eqref{S} is close to the stable equilibrium $(0,\bar{n}_a)$:
\begin{equation}\label{deftepsz1} t_{\eps}(z):=\inf \big\{ s \geq 0,\forall t \geq s, ({n}_A^{(z)}(t),{n}_a^{(z)}(t))\in 
[0,\eps^2/2]\times[\bar{n}_a-\eps/2,\infty) \big\}. \end{equation}
Under condition \eqref{condfitnesses} $n^{(z)}$ never escapes from the set $[0,\eps^2/2]\times[\bar{n}_a-\eps/2,\infty)$
once it has reached  the latter. Moreover, $t_\eps(z)$ is finite and, 
\begin{equation}\label{finitudetepseta}
 t^{(\eps)}:= \sup \{ t_\eps(z), z \in [\eps/2, \eps] \times  [\bar{n}_A-2\eps C_{A,a}/C_{A,A},\bar{n}_A+2\eps C_{A,a}/C_{A,A}] \}<\infty.
\end{equation}

Recall that $\beta_2=0$ and $\beta_3=\beta$.
Combining Lemma \ref{duration}, and Equations \eqref{compST} 
and \eqref{finitudetepseta}, we get under Assumption $i$ for 
$i  \in \{2,3\}$ that after a time of order 
$ (1-\beta_i) \log K/S_{aA} $, the $A$-population size is close to $ \eps^2 K/2$ and the $a$-population size is close to 
$\bar{n}_a K$. Under Assumptions \ref{assmuK2} and \ref{assmuK3}, $\lambda^{\alpha \bar{\alpha}}, \alpha \in \mathcal{A}$, are negligible with respect to $1$. As a consequence, we can use the proof in \cite{champagnat2006microscopic} to get that the time needed 
to end the sweep after time $T_\eps^K+ t^{(\eps)}$ is close to 
$\log K/|S_{Aa}|$ with high probability. This proof lies on a coupling of the $A$-population dynamics with this of a subcritical birth and death process with 
individual birth rate $f_A$ and individual death rate $D_A+C_{A,a}\bar{n}_a= f_A+|S_{Aa}|$.
This ends the proof of Theorem \ref{theoduration}.
\end{proof}

\section{Proof of Theorem \ref{theoprecise}} \label{sectionprecise}

The first point of Theorem \ref{theoprecise} has already been stated in \cite{champagnat2006microscopic} and we will focus 
on Assumptions \ref{assmuK2} and \ref{assmuK3}.
We assume along this section that \eqref{condfitnesses} and \eqref{cond_ini} hold, and refer to Section \ref{sectioncoupl2} 
for the definitions 
of the different processes coupled with the population process $N$.

\subsection{Proof of Theorem \ref{theoprecise} point (2)}

We suppose in this subsection that Assumption \ref{assmuK2} holds.
The idea is to cut the first phase into two parts: the first part, which has a duration $|\ln \eps|/4S_{aA}$, is long enough for a sufficient number of mutants to appear, 
and short enough to construct efficient couplings of processes $(N_i, i \geq 1)$ with birth and death processes.
We then show that the relative sizes of the populations generated by the successive $a$-mutants stay approximately the same until the end of the sweep.

For the sake of readability, we introduce the time
\begin{equation} \label{defLeps} l_\eps:=  |\ln \eps|/4S_{aA},\end{equation}
and we recall the definitions of $(T_i^{*}, Z_i^{*}, i \geq 1, * \in \{-,+,0\})$ in \eqref{defTi*} and \eqref{defZi*}.

We want to control the relative sizes of the populations generated by the successive $a$-mutants at time $l_\eps$. 
If we define the event
\begin{equation}\label{defegalT(i)}
 Tid:=\{ T_i^{-}=T_i^{0}=T_i^{}=T_i^{+}, \forall i \text{ such that } T_i^{+}\leq l_\eps \},
\end{equation}
we get
\begin{eqnarray*}
\P(Tid)=\sum_{k \geq 0} \P\Big(N_m^+( l_\eps)=k, Tid\Big) 
&\geq & e^{- \lambda_K^{+}(\eps) l_\eps} \sum_{k \geq 0} \frac{(\lambda_K^{+}(\eps) l_\eps)^k}{k!} 
\Big(\frac{\lambda_K^{-}(\eps)}{\lambda_K^{+}(\eps)}\Big)^k\\
&=& \exp\left(\left(\lambda_K^{-}(\eps)- \lambda_K^{+}(\eps)\right) l_\eps\right)\geq 1- \eps^{1/4}
\end{eqnarray*}
for $\eps$ small enough and $K$ large enough, where $\lambda_K^{-}(\eps)/\lambda_K^{+}(\eps)$ is the probability that the Poisson process
$N_m^-$ jumps at a given time conditionally on having a jump of the Poisson process $N_m^+$ at the same time.
Moreover, we have from the definition of $Z^{-}$ in \eqref{defNi} and thanks to \eqref{ext_times},
\begin{eqnarray*}
\P( Z^{-}(l_\eps)=0)&=& \sum_{k=0}^{\infty} \P( Z^{-}(l_\eps)=0,N_m^{-}(l_\eps)=k)\\
&\leq &e^{-\lambda_K^{-}(\eps)l_\eps}\sum_{k=0}^{\infty} \frac{1}{k!}
\left(\lambda_K^{-}(\eps)l_\eps(1-s_-(\eps))\right)^k\\
&=& e^{-s_-(\eps)\lambda_K^{-}(\eps)l_\eps}= \eps^{s_-(\eps)\lambda_K^{-}(\eps)/4S_{aA}}.
\end{eqnarray*}
From the two last inequalities we get the existence of two positive constants $c$ and $\gamma$ such that for $\eps$ small enough,
\begin{equation}\label{evtprobable}
\liminf_{K \to \infty} \P(Tid,Z^{-}(l_\eps)\geq 1)\geq 1- c\eps^\gamma.
\end{equation}

By Couplings \eqref{couplageNi} and \eqref{couplageZ0} we have almost surely on the event $Tid$ and for every $i\geq 1$,
\begin{equation} \label{comparaisonproportions} \frac{Z_i^{-}(l_\eps)}{Z^{+}(l_\eps)}\leq 
\frac{Z_i^{0}(l_\eps)}{Z^{0}(l_\eps)}, \frac{N_i(l_\eps)}{N_a(l_\eps)} \leq
\frac{Z_i^{+}(l_\eps)}{Z^{-}(l_\eps)} .\end{equation}
Hence, if we are able to show that $(Z_i^{+}/Z^{-})(l_\eps) $ and 
$(Z_i^{+}/Z^{-})(l_\eps) $ are close for every $i$, it will imply that 
$(Z_i^{0}/Z^{0})(l_\eps)$ and $(N_i/N_a)(l_\eps) $ 
are close for every $i$.
Let us first notice that for every $t \geq 0$, the following inequality holds
\begin{eqnarray*}
\mathbf{1}_{\{ Z^{-}(t)\geq 1,Tid \}}\left( \frac{Z^{+}(t)}{Z^{-}(t)}- \frac{Z^{-}(t)}{Z^{+}(t)}\right)
&=&\mathbf{1}_{\{ Z^{-}(t)\geq 1 ,Tid\}} \frac{(Z^{+}(t)-Z^{-}(t))(Z^{+}(t)+Z^{-}(t))}{Z^{-}(t)Z^{+}(t)}\\
&\leq & 2\mathbf{1}_{\{ Tid \}}(Z^{+}(t)-Z^{-}(t)).\end{eqnarray*}
In order to use Lemma \ref{lemcompmart}, we bound the last term as follows:
\begin{multline*}
\mathbf{1}_{\{ Tid \}}(Z^{+}(t)-Z^{-}(t))\leq  Z^{-}(t)\left( e^{f_a(s_+(\eps)-s_-(\eps))t} -1\right)+\\
\mathbf{1}_{\{ Tid \}}
\sum_{i = 1}^{N_m^{+}(t)}e^{f_as_+(\eps)(t-T_i^{0})}\left(Z_i^{+}(t)e^{-f_a(s_+(\eps)-\mu_K^{aA})(t-T_i^{0})}-Z_i^{-}(t)e^{-f_a(s_-(\eps)-
\mu_K^{aA})(t-T_i^{0})}\right) .
\end{multline*}

We are now able to compare the relative mutant population sizes. Applying \eqref{evtprobable} and the two last remarks, we get
\begin{multline} \label{longineq}
  \P\Big(\exists i, \frac{Z_i^{+}(l_\eps)}{Z^{-}(l_\eps)}- \frac{Z_i^{-}(l_\eps)}{Z^{+}(l_\eps)}>\eps^{1/8}\Big)\\
\leq c\eps^\gamma+ \P\Big(2(Z^{+}(l_\eps)-Z^{-}(l_\eps))>\eps^{1/8}, Tid \Big)\\
\leq c\eps^\gamma+ 2\eps^{-1/8} \E[\mathbf{1}_{\{ Tid \}}(Z^{+}(l_\eps)-Z^{-}(l_\eps))]\\
\leq c\eps^\gamma+ 2\eps^{-1/8}\Big\{ \left( e^{f_a(s_+(\eps)-s_-(\eps))l_\eps} -1\right)\E[Z^{-}(l_\eps)]+ 
e^{f_as_+(\eps)l_\eps} \\
 \E\Big[ \E\Big[\mathbf{1}_{\{ Tid \}} \sum_{i=1}^{N_m^{+}(l_\eps)} 
\left|Z_i^{+}(t)e^{-f_a(s_+(\eps)-\mu_K^{aA})(l_\eps-T_i^{0})}-Z_i^{-}(t)e^{-f_a(s_-(\eps)-\mu_K^{aA})(l_\eps-T_i^{0})}\right| \\
\Big| (N_m^{+}(t), t \leq l_\eps) \Big]\Big]\Big\}.
  \end{multline}
But we know that conditionally on $ Tid $ and on $(N_m^{+}(t), t \leq l_\eps) $, the random variables 
$$ Z_i^{+}(t)e^{-f_a(s_+(\eps)-\mu_K^{aA})(t-T_i^{0})}-Z_i^{-}(t)e^{-f_a(s_-(\eps)-\mu_K^{aA})(t-T_i^{0})} $$
are independent martingales with zero mean, and according to Lemma \ref{lemcompmart} 
we can couple each $(Z_i^{+},Z_i^{-})$ with a pair of processes $(\tilde{Z}_i^{+},\tilde{Z}_i^{-})$
such that $Z_i^*$ and $\tilde{Z}_i^*$ have the same distribution for $* \in \{-,+\}$, and
 $$ \sup_{t \geq 0} \E\Big[ \Big( \tilde{Z}_i^{+}(t)e^{-f_a(s_+(\eps)-\mu_K^{aA})(t-T_i^{0})}-\tilde{Z}_i^{-}(t)
 e^{-f_a(s_-(\eps)-\mu_K^{aA})(t-T_i^{0})} \Big)^2 \Big]\leq s_+(\eps)-s_-(\eps) .$$
As a consequence, recalling Definition \eqref{def_s_-s_+1} and using Cauchy-Schwarz Inequality, the last expectation in \eqref{longineq} can be bounded by
$$ \E\Big[  \Big( N_m^{+}(l_\eps) (s_+(\eps)-s_-(\eps)) \Big)^{1/2} \Big]\leq  \lambda_K^{+}(\eps) (s_+(\eps)-s_-(\eps))^{1/2}\leq c \eps^{1/2} $$
for $\eps$ small enough and $K$ large enough, where $c$ is a finite constant.
  
Furthermore, the first expectation in \eqref{longineq} can be computed explicitely. Let $\delta$ be in $(0,1)$. 
Then there exists a finite $c$ such that for $\eps$ small enough
\begin{multline*}
\left( e^{f_a(s_+(\eps)-s_-(\eps))l_\eps} -1\right)\E[Z^{-}(l_\eps)]\\
=\left( e^{f_a(s_+(\eps)-s_-(\eps))l_\eps} -1\right)\frac{\lambda_K^{-}(\eps)}{(s_-(\eps)-\mu_K^{aA})f_a}
\Big(e^{(s_-(\eps)-\mu_K^{aA})f_al_\eps}-1\Big)\\
\leq c\eps^{1-\delta} \eps^{-1/4}=c\eps^{3/4-\delta}.
\end{multline*}
Adding that for $\eps$ small enough and $K$ large enough, 
$$ e^{f_as_+(\eps)l_\eps}\leq \eps^{-1/4-\delta} $$
and taking $\delta=1/16$, we conclude that for $\eps$ small enough,
\begin{equation} \label{majdiffproportions}
  \limsup_{K \to \infty} \P\Big(\exists i, \frac{Z_i^{+}(l_\eps)}{Z^{-}(l_\eps)}- \frac{Z_i^{-}(l_\eps)}{Z^{+}(l_\eps)}>\eps^{1/8}\Big)
 \leq c(\eps^\gamma+\eps^{1/16}),
\end{equation}
where $c$ is a finite constant.

Let us now consider for every $t \geq 0$ and $* \in \{-,+,0\}$ the subsequence of mutant populations which survive until a given time:
$$ (Z_{i,t}^{(*)}, i \geq 1)= (Z_{j}^{*}, j \geq 1 \text{ and } Z_{j}^{*}(t)\geq 1), $$
where the index $i$ increases with the birth time of the mutant at the origin of the surviving population.
We introduce similarly the sequence $(N_{i,t}^{()}, i \geq 1)$ for surviving processes $N_i$.

For every $i$ such that $T_i^-=T_i^+ \leq l_\eps$, the probability that $Z_i^+$ survives until time $l_\eps$ and 
$Z_i^-$ does not survive until time $l_\eps$ is bounded by $c \eps $ for a finite $c$ (thanks to \eqref{ext_times} and because $\lambda_K^+(\eps)- \lambda_K^-( \eps)$ 
is of order $\eps$). Adding \eqref{comparaisonproportions} and \eqref{majdiffproportions}, we finally get
the existence of three positive constants $(\gamma_1,\gamma_2,c)$ such that, 
\begin{equation}\label{majetape1}
 \P\Big(\Big\|\Big(\frac{N_{i,l_\eps}^{()}(l_\eps)}{N_a(l_\eps)}, i \geq 1\Big)- \Big(\frac{Z_{i,l_\eps}^{(0)}(l_\eps)}{Z^{(0)}(l_\eps)}, i \geq 1\Big)\Big\| >\eps^{\gamma_1} \Big)
 \leq c\eps^{\gamma_2}, 
\end{equation}
 where $\|x\| = \sup_i |x_i|$.

To study the evolution of the relative population sizes during the end of the first phase we will 
proceed in two steps: first we will show that the $a$-population at time $T_\eps^K$ is essentially constituted by descendants 
of $a$-mutants born from $A$-individuals before the time $l_\eps$; 
then we will show that the relative sizes of the $a$-populations generated by the $a$-mutants born before the time $l_\eps$ 
stay approximately the same during the time 
interval $[l_\eps,T_\eps^K]$.
Hence as a first step we prove that there exists a finite $c$ such that for $\eps$ small enough, $K$ large enough, and $t \geq l_\eps$,
\begin{equation}\label{firststep}
  \E\Big[ \frac{\bar{N}_a(t\wedge {T}_\eps^K\wedge S_\eps^K)+1}{N_a(t\wedge {T}_\eps^K\wedge S_\eps^K)+1} \Big] \geq
 1- c\int_{l_\eps}^t \E\Big[ \frac{\mathbf{1}_{\{N_a(s\wedge {T}_\eps^K\wedge S_\eps^K)\geq1\}}}{N_a(s)+1} \Big]ds  ,
\end{equation}
where $\bar{N}_a$ denotes the population generated by the $a$-mutants born before time $l_\eps$.
To prove this inequality, we use an equivalent representation of the $a$-population. Let $Q(ds,d\theta)$ and $R(ds,d\theta)$ 
be two independent Poisson random measures with intensity $ds d\theta$, independent of $Q_m$, and recall Definition \eqref{deathrate}. 
Then on the time 
interval $[l_\eps,T_\eps^K]$, the triplet $(\bar{N}_a,N_a,N_A)$ has the same law as 
$(\tilde{\bar{N}}_a,\tilde{N}_a,\tilde{N}_A)$ defined by:

\begin{multline*}
 \tilde{\bar{N}}_a(t) = N_a(l_\eps) + \int_{l_\eps}^t \int_{\R_+} Q(ds,d\theta)\Big( \mathbf{1}_{\theta \leq (1-\mu_K^{aA})f_a \tilde{\bar{N}}_a(s^-)}\\
 - \mathbf{1}_{0<\theta-(1-\mu_K^{aA})f_a \tilde{\bar{N}}_a(s^-) \leq D_a^K(\tilde{N}(s^-)) \tilde{\bar{N}}_a(s^-)}
 \Big)
=  N_a(l_\eps) + \int_{l_\eps}^t \int_{\R_+} f(s^-) Q(ds,d\theta),
\end{multline*}
\begin{multline*}
 (\tilde{N}_a-\tilde{\bar{N}}_a)(t) =  
  \int_{l_\eps}^t \int_{\R_+} \Big( \mathbf{1}_{0<\theta-((1-\mu_K^{aA})f_a+D_a^K(\tilde{N}(s^-))) \tilde{\bar{N}}_a(s^-) 
 \leq (1-\mu_K^{aA})f_a (\tilde{N}_a-\tilde{\bar{N}}_a)(s^-)}
\\ 
- \mathbf{1}_{0<\theta -((1-\mu_K^{aA})f_a \tilde{N}_a(s^-) +D_a^K(\tilde{N}(s^-)) \tilde{\bar{N}}_a)(s^-)
\leq D_a^K(\tilde{N}(s^-)) (\tilde{N}_a-\tilde{\bar{N}}_a)(s^-)}
 \Big)Q(ds,d\theta)\\
 +  \int_{l_\eps}^t  \int_{\R_+} \mathbf{1}_{\theta \leq \mu_K^{Aa}f_A \tilde{N}_A(s^-)}Q_m(ds,d\theta)  
=   \int_{l_\eps}^t \int_{\R_+} \bar{f}(s^-) Q(ds,d\theta)
+ \int_{l_\eps}^t \int_{\R_+} g(s^-) Q_m(ds,d\theta),
\end{multline*}
and
\begin{multline*} \tilde{N}_A(t) = N_A(l_\eps) + \int_{l_\eps}^t \int_{\R_+} \Big( \mathbf{1}_{\theta \leq (1-\mu_K^{Aa})f_A 
\tilde{N}_A(s^-)+ f_a \mu_K^{aA}\tilde{N}_a(s^-)}\\
 - \mathbf{1}_{0<\theta-(1-\mu_K^{Aa})f_a \tilde{N}_A(s^-)-f_a \mu_K^{aA}\tilde{N}_a(s^-) \leq D_A^K(\tilde{N}(s^-)) \tilde{N}_A(s^-)}
 \Big)R(ds,d\theta).\end{multline*}
We get from Itô's formula for $t \geq l_\eps$,
 \begin{multline}\label{itofi}
  \frac{\tilde{\bar{N}}_a(t)+1}{\tilde{N}_a(t)+1} =  1
-  \int_{l_\eps}^{t} \int_{\R_+}   \frac{g(s^-)(\tilde{\bar{N}}_a(s^-)+1)}
  {(\tilde{N}_a(s^-)+1)(\tilde{N}_a(s^-)+2)}  Q_m(ds,d\theta)\\
  + 
  \int_{l_\eps}^{t} \int_{\R_+}  \frac{f(s^-)(\tilde{N}_a-\tilde{\bar{N}}_a)(s^-)-\bar{f}(s^-)\tilde{\bar{N}}_a(s^-)}
  {(\tilde{N}_a(s^-)+1)(\tilde{N}_a(s^-)+1+f(s^-)+\bar{f}(s^-))}  (Q(ds, d\theta)-ds d\theta) .
 \end{multline}

 Taking the expectation at time $t\wedge {T}_\eps^K\wedge S_\eps^K$ yields
 \begin{multline*}
  \E\Big[ \frac{\bar{N}_a(t\wedge {T}_\eps^K\wedge S_\eps^K)+1}{N_a(t\wedge {T}_\eps^K\wedge S_\eps^K)+1} \Big] \geq 
1- \int_{l_\eps}^t \E\Big[ \mathbf{1}_{\{N_a(s\wedge {T}_\eps^K\wedge S_\eps^K)\geq1\}}\frac{\mu_K^{Aa}f_A N_A(s)}{N_a(s)+1} \Big]ds \\
\geq 1- 
\mu_K^{Aa}f_A K\Big(\bar{n}_A-2\eps \frac{C_{A,a}}{C_{A,A}}\Big)
\int_{l_\eps}^t \E\Big[ \frac{\mathbf{1}_{\{N_a(s\wedge {T}_\eps^K\wedge S_\eps^K)\geq1\}}}{N_a(s)+1} \Big]ds,\end{multline*}
which implies \eqref{firststep} under Assumption \ref{assmuK2}.
 
Reasoning similarly as in the proof of Lemma 3.1 in \cite{smadi2014eco}, we get the existence of two integers $(k_0,k_1)$ and two pure jump 
martingales $(M_0,M_1)$, such that for $s \geq 0$,
\begin{equation}  \label{defk0}
\frac{\mathbf{1}_{\{N_a(s\wedge {T}_\eps^K\wedge S_\eps^K)\geq 1\}} }{N_a(s\wedge {T}_\eps^K\wedge S_\eps^K)+k_0}e^{\frac{S_{aA}(s\wedge {T}_\eps^K\wedge S_\eps^K)}{2(k_0+1)}}
\leq (k_0+1)M_0(s\wedge {T}_\eps^K\wedge S_\eps^K),\quad
 \E[M_0(s)]= \frac{1}{k_0+1},
\end{equation}
and
\begin{equation}  \label{defk1}
\frac{\mathbf{1}_{\{\bar{N}_a(s\wedge {T}_\eps^K\wedge S_\eps^K)\geq 1\}} }{\bar{N}_a(s\wedge {T}_\eps^K\wedge S_\eps^K)+k_1}
e^{\frac{S_{aA}(s\wedge {T}_\eps^K\wedge S_\eps^K)}{2(k_1+1)}}
\leq (k_1+1)M_1(s\wedge {T}_\eps^K\wedge S_\eps^K),\quad
 \E[M_1(s)]= \frac{1}{k_1+1}.
\end{equation}

Hence we get for a finite $c$ and $K$ large enough,
\begin{eqnarray*}
  \E\Big[ \frac{\bar{N}_a(t\wedge {T}_\eps^K\wedge S_\eps^K)+1}{N_a(t\wedge {T}_\eps^K\wedge S_\eps^K)+1} \Big]& \geq &
 1-c
\int_{l_\eps}^t  e^{-S_{aA}s/(2(k_0+1))} 
\E\Big[ \mathbf{1}_{s\leq T_\eps^K\wedge S_\eps^K}\frac{e^{S_{aA}s/(2(k_0+1))}}{N_a(s)+1} \Big]ds\\
&\geq &1- 
c
\int_{l_\eps}^t  e^{-S_{aA}s/(2(k_0+1))} 
\E\Big[ \frac{e^{S_{aA}(s\wedge {T}_\eps^K\wedge S_\eps^K/(2(k_0+1)))}}
{N_a(s\wedge T_\eps^K\wedge S_\eps^K)+1} \Big]ds\\
&\geq &1- 
c
\int_{l_\eps}^t  e^{-S_{aA}s/(2(k_0+1))} ds
\geq 1- 
c \eps^{1/(8(k_0+1))} .
\end{eqnarray*}
Markov Inequality then yields
\begin{equation} \label{seulsavantLepscomptent}
\limsup_{K \to \infty} \P\Big( 1- \frac{\bar{N}_a( T_\eps^K\wedge S_\eps^K)+1}{N_a( T_\eps^K\wedge S_\eps^K)+1} >
\eps^{1/(16(k_0+1))}\Big)\leq c \eps^{1/(16(k_0+1))} ,\end{equation}
for $\eps$ small enough, where $c$ is a finite constant.
As a consequence to get the distribution of the relative $a$-population sizes at time $T_\eps^K$ it is enough to focus on the evolution of the 
processes $((N_{i,l_\eps}^{()}+1)/(\bar{N}_a+1), i \geq 1)$ over the time interval $[l_\eps,T_\eps^K]$, where we recall that $N_{i,l_\eps}^{()}$ is the $i$th $a$-mutant population
which survives at least until time $l_\eps$.
Applying Itô's formula as in \eqref{itofi}, we get that the processes
$$\left(M^{(i)}(t):=\frac{N_{i,l_\eps}^{()}(l_\eps+t)+1}{\bar{N}_a(l_\eps+t)+1}, t \geq 0\right), \quad i \geq 1$$
are martingales, whose quadratic variations satisfy
\begin{multline} \label{majvarquad}
 \frac{d}{dt}\E[\langle M^{(i)}\rangle_t^2]= 
 \E \Big[f_a N_{i,l_\eps^K}^{()}(l_\eps+t) \Big(\frac{(\bar{N}_a-N_{i,l_\eps^K}^{()})(l_\eps+t)}{(\bar{N}_a(l_\eps+t)+1)(\bar{N}_a(l_\eps+t)+2)}\Big)^2\\
 +D_a^K(N(l_\eps+t))N_{i,l_\eps^K}^{()}(l_\eps+t) \Big(\frac{(\bar{N}_a-N_{i,l_\eps^K}^{()})(l_\eps+t)}{\bar{N}_a(l_\eps+t)(\bar{N}_a(l_\eps+t)+1)}\Big)^2\\
 +f_a (\bar{N}_a-N_{i,l_\eps^K}^{()})(l_\eps+t) \Big(\frac{N_{i,l_\eps^K}^{()}(l_\eps+t)}{(\bar{N}_a(l_\eps+t)+1)(\bar{N}_a(l_\eps+t)+2)}\Big)^2\\
 +D_a^K(N(l_\eps+t)) (\bar{N}_a-N_{i,l_\eps^K}^{()})(l_\eps+t) \Big(\frac{N_{i,l_\eps^K}^{()}(l_\eps+t)}{\bar{N}_a(l_\eps+t)(\bar{N}_a(l_\eps+t)+1)}\Big)^2\Big]\\
 \leq \E \Big[\frac{f_a }{\bar{N}_a(l_\eps+t)+1}\Big(\frac{(\bar{N}_a-N_{i,l_\eps^K}^{()})(l_\eps+t)}{\bar{N}_a(l_\eps+t)+2}\Big)^2
+ \frac{D_a^K(N(l_\eps+t))}{\bar{N}_a(l_\eps+t)+1} \Big(\frac{N_{i,l_\eps^K}^{()}(l_\eps+t)}{\bar{N}_a(l_\eps+t)}\Big)^2\Big]\\
 \leq \E \Big[\mathbf{1}_{\{\bar{N}_a(l_\eps+t)\geq 1\}}\frac{f_a +D_a^K(N(l_\eps+t))}{\bar{N}_a(l_\eps+t)+1}\Big].
\end{multline}
From the definition of $D_\alpha^K$, and $T_\eps^K$ and $S_\eps^K$ in \eqref{deathrate} and \eqref{TKTKeps1}, respectively, we see that for $ t\leq T_\eps^K \wedge S_\eps^K$, 
$$  D_a^K(N(t))\leq D_a+ C_{a,A}\left( \bar{n}_a+ 2\eps \frac{C_{A,a}}{C_{A,A}} \right)+C_{a,a}\eps .$$
As a consequence, a direct application of \eqref{defk1} yields
the existence of a positive constant such that for $\eps$ small enough and $t \geq 0$,
\begin{multline}\label{boundthird}
 \E \left[ \left( \frac{N_{i,l_\eps^K}^{()}(T_\eps^K\wedge S_\eps^K\wedge t)+1 }{\bar{N}_a(T_\eps^K\wedge S_\eps^K\wedge t)+1}
-\frac{N_{i,l_\eps^K}^{()}(l_\eps)+1 }{\bar{N}_a(l_\eps)+1} \right)^2 \right] \\ \leq 
c \int_{l_\eps}^\infty  e^{-S_{aA}s/(2(k_1+1))} 
\E\Big[ \frac{e^{S_{aA}(s\wedge {T}_\eps^K\wedge S_\eps^K/(2(k_1+1)))}}
{\bar{N}_a(s\wedge T_\eps^K\wedge S_\eps^K)+k_1} \Big]ds\leq c\eps^{1/(8(k_1+1))}.
\end{multline}

To end the proof of Theorem \ref{theoprecise}, we need to bound the term
$$ \Big\|\Big(\frac{N_{i,T_\eps^K \wedge S_\eps^K}^{()}(T_\eps^K \wedge S_\eps^K)}
{N_a(T_\eps^K \wedge S_\eps^K)}, i \geq 1\Big)- \frac{Z_{i,l_\eps^K}^{(0)}(l_\eps)}
{Z^{(0)}(l_\eps )}\Big\|, $$
and to show that the second term of the difference has a distribution close to the GEM distribution (defined in Definition \ref{defiGEM}) with the desired parameter.
But from inequalities \eqref{compST}, \eqref{seulsavantLepscomptent} and \eqref{boundthird}, we see that 
with high probability it amounts to focus on the difference 
$$ \Big\|\Big(\frac{N_{i,l_\eps^K}^{()}(T_\eps^K)}
{N_a(T_\eps^K )}, i \geq 1\Big)- (\frac{Z_{i,l_\eps^K}^{(0)}(l_\eps)}
{Z^{(0)}(l_\eps)}, i \geq 1)\Big\|. $$
Let $i$ be smaller than $\#\{j \geq 1, N_j (T_\eps^K)\geq 1\}$. Then, on the 
event $\{\bar{N}_a(T_\eps^K) \geq 1\}$ (which implies that $N_a(T_\eps^K )=\lfloor \eps K \rfloor$), the triangle inequality yields
\begin{multline*}
\left| \frac{N_{i,l_\eps^K}^{()}(T_\eps^K)}
{N_a(T_\eps^K )}- \frac{Z_{i,l_\eps^K}^{(0)}(l_\eps)}
{Z^{(0)}(l_\eps)}\right| \leq \left| \frac{N_{i,l_\eps^K}^{()}(T_\eps^K)}
{N_a(T_\eps^K )}- \frac{N_{i,l_\eps^K}^{()}(T_\eps^K)+1}
{N_a(T_\eps^K )+1}\right|+ \left| \frac{N_a(T_\eps^K)-\bar{N}_a(T_\eps^K) }
{N_a(T_\eps^K )+1}\right|\\
+ \left| \frac{N_{i,l_\eps^K}^{()}(T_\eps^K)+1}
{\bar{N}_a(T_\eps^K )+1}-  \frac{N_{i,l_\eps^K}^{()}(l_\eps)+1}
{\bar{N}_a(l_\eps )+1}\right|
+ \left| \frac{N_{i,l_\eps^K}^{()}(l_\eps^K)+1}
{\bar{N}_a(l_\eps^K )+1}-  \frac{N_{i,l_\eps^K}^{()}(l_\eps)}
{\bar{N}_a(l_\eps )}\right|
+ \left| \frac{N_{i,l_\eps^K}^{()}(l_\eps)}
{\bar{N}_a(l_\eps )}-\frac{Z_{i,l_\eps^K}^{(0)}(l_\eps)}
{Z^{(0)}(l_\eps)}\right|\\
\leq   \frac{1}{\lfloor \eps K \rfloor+1}+ \left| \frac{N_a(T_\eps^K)-\bar{N}_a(T_\eps^K) }
{N_a(T_\eps^K )}\right|
+ \left| \frac{N_{i,l_\eps^K}^{()}(T_\eps^K)+1}
{\bar{N}_a(T_\eps^K )+1}-  \frac{N_{i,l_\eps^K}^{()}(l_\eps)+1}
{\bar{N}_a(l_\eps )+1}\right|\\ 
+  \frac{1}{N_a(l_\eps^K )+1}
+ \left| \frac{N_{i,l_\eps^K}^{()}(l_\eps)}
{N_a(l_\eps )}-\frac{Z_{i,l_\eps^K}^{(0)}(l_\eps)}
{Z^{(0)}(l_\eps)}\right|.
\end{multline*}
The second term can be bounded thanks to \eqref{seulsavantLepscomptent}, 
the third one thanks to \eqref{boundthird}, the fourth one thanks to \eqref{defk0}, and the last one 
thanks to \eqref{majetape1}. As there is a number of order $\ln \eps$ of surviving mutant populations at time $l_\eps$,
we conclude by using Markov Inequality that there exist 
three positive constants $(\gamma_3,\gamma_4,c)$ such that for $\eps$ small enough, 
\begin{equation}\label{majetape2}
 \P\Big(\Big\|\Big(\frac{N_{i,T_\eps\wedge S_\eps^K}^{()}(T_\eps\wedge S_\eps^K)}{N_a(T_\eps\wedge S_\eps^K)}, i \geq 1\Big)-
 \Big(\frac{Z_{i,l_\eps^K}^{(0)}(l_\eps)}
{Z^{(0)}(l_\eps)}, i \geq 1\Big)\Big\| >\eps^{\gamma_3} \Big)
 \leq c\eps^{\gamma_4},
\end{equation}
which is equivalent to the following convergence in probability:
\begin{equation}\label{convproba}
 \Big(\frac{N_{i,T_\eps\wedge S_\eps^K}^{()}(T_\eps\wedge S_\eps^K)}{N_a(T_\eps\wedge S_\eps^K)}, i \geq 1\Big)-
 \Big(\frac{Z_{i,l_\eps^K}^{(0)}(l_\eps)}
{Z^{(0)}(l_\eps)}, i \geq 1\Big) \underset{K \to \infty}{\to} 0.
\end{equation}
But a direct application of Theorem 2.2 in \cite{richard2011limit} gives the behavior of the relative mutant population sizes in the process $Z^{(0)}$:
\begin{equation} \label{resultRichard} \underset{\eps \to 0}{\lim}(Z^{(0)}(l_\eps))^{-1}\Big(Z^{(0)}_{1,l_\eps}(l_\eps), 
Z^{(0)}_{2,l_\eps}(l_\eps),...\Big)=(P_1,P_2,...) \quad \text{a.s.},\end{equation}
where the sequence $(P_1,P_2,...)$ has a GEM distribution with parameter $\lambda_K^{0}/(1-\mu_K^{aA})$ and $\lambda_K^{0}$ has been defined in \eqref{deflambda*}.
Hence, using Slutsky's theorem we get
\begin{equation} \label{convloi1}
 \Big(\frac{N_{i,T_\eps\wedge S_\eps^K}^{()}(T_\eps\wedge S_\eps^K)}{N_a(T_\eps\wedge S_\eps^K)}, i \geq 1\Big) \underset{K \to \infty}{\to} (P_1,P_2,...) \quad \text{in law}.
\end{equation}

The next step of the proof consists in checking that during the "deterministic phase", the relative 
mutant population sizes do not vary. Let $i$ be smaller than $\#\{j \geq 1, N_j (T_\eps^K)\geq 1\}$.
Then applying again Theorem 2.1 p. 456 in Ethier and Kurtz \cite{ethiermarkov}, we can show that on the event $T_\eps^K \leq S_\eps^K$, 
$$(N_{i,T_\eps\wedge S_\eps^K}^{()}/K,N_a/K,N_A/K)$$
converges in probability on the time 
interval $[T_\eps^K, T_\eps^K+t^{(\eps)}]$ (defined in \eqref{finitudetepseta}) 
to the solution of the deterministic system:
$$ \left\{ \begin{array}{l}
    \dot{n}_A = (f_A -D_A -C_{A,A}n_A-C_{A,a}n_a)n_A \\
    \dot{n}_a^{(i)} = (f_a -D_a -C_{a,A}n_A-C_{a,a}n_a)n_a^{(i)} \\
    \dot{n}_a = (f_a -D_a -C_{a,A}n_A-C_{a,a}n_a)n_a ,
    \end{array}\right.
 $$
 with initial condition
 $$(N_{i,T_\eps\wedge S_\eps^K}^{()}(T_\eps^K)/K,N_a(T_\eps^K)/K,N_A(T_\eps^K)/K).$$
In particular, 
$$ \frac{d}{dt} \left( \frac{n_a^{(i)}}{n_a} \right)=0 .$$
Now, if we introduce
\begin{equation*} \label{defTFKeta}
 T_F^K(\eps):= \inf \{t \geq 0, N_A^K(t)\leq  \eps^2 K/2, N_a^K(t)\geq K(\bar{n}_a-\eps/2) \},
\end{equation*}
then for $\eps$ small enough,
$$ \lim_{K \to \infty}\P( T_\eps^K+t^{(\eps/2)}\geq T_F^K(\eps) |T_\eps^K \leq S_\eps^K)=1, $$
and moreover, after the time $T_F^K(\eps)$ the $A$-population size stays, for a very long time and with a 
probability close to one, smaller than $\eps$. More precisely following \cite{champagnat2006microscopic}, we get the existence of $c,\eps_0, V>0$, 
such that for every $\eps \leq \eps_0$,
\begin{equation} \label{expKV} \limsup_{K \to \infty}\P\left(\sup_{T_F^K(\eps)\leq t \leq e^{KV}} 
\Big\|\frac{1}{K} (N_A^K(t),N_a^K(t))-(0,\bar{n}_a) \Big\| \leq \eps \right) \leq c \eps .\end{equation}

The last step of the proof consists in showing that the fractions 
$$ F_i(t):=\frac{N_{i,T_F^K(\eps)}^{()}(T_F^K(\eps)+t)}{N_a(T_F^K(\eps)+t)} $$
stay almost constant during a time $Gf_k$, where the constant $G$ and the function $f_K$ have been introduced in Theorem \ref{theoprecise}.
The proof of this fact is similar to the proof of Proposition $1$ in \cite{smadi2014eco}. It consists in applying Itô's formula with jumps
(see \cite{ikeda1989stochastic} p. 66) to compute expectation and quadratic variation of $F_i(t)$. We get
$$ \E \left[ F_i(t)- F_i(0) \right]= - f_A \lambda^{Aa} \int_0^t \mathbf{1}_{N_a(T_F^K(\eps)+s)\geq 1} 
\left(\frac{N_A N_{i,T_F^K(\eps)}^{()}}{N_a(N_a+1)}\right)(T_F^K(\eps)+s) $$
and 
\begin{multline*}
 \langle F_i \rangle_t= \int_0^t \Big\{ \mathbf{1}_{\{N_a(T_F^K(\eps)+s)\geq 1\}}f_a(1- \lambda^{aA})
 \left( \frac{N_{i,T_F^K(\eps)}^{()}(N_a-N_{i,T_F^K(\eps)}^{()})}{N_a(N_a+1)^2} \right)(T_F^K(\eps)+s)
 \\ + \mathbf{1}_{\{N_a(T_F^K(\eps)+s)\geq 1\}}f_A \lambda{Aa} \left(
 \frac{N_A(N_a-N_{i,T_F^K(\eps)}^{()})^2}{N_a^2(N_a+1)^2}  \right)(T_F^K(\eps)+s)\\
+ \mathbf{1}_{\{N_a(T_F^K(\eps)+s)\geq 2\}}D_a^K(N(T_F^K(\eps)+s))\left(
\frac{N_{i,T_F^K(\eps)}^{()}(N_a-N_{i,T_F^K(\eps)}^{()})}{N_a(N_a-1)^2} \right)(T_F^K(\eps)+s)\Big\}ds,
\end{multline*}
where $D_a^K$ has been defined in \eqref{deathrate}.
Adding \eqref{expKV} allows us to conclude that with a probability close to $1$, 
$F_i(t)$ stays almost constant during any time negligible with respect to $K$.
Using the fact that for a given $\eps$, a number of mutant subpopulations large but independent from $K$ is enough to constitute a fraction 
$1-\eps$ of the $a$-population at time $T_\eps^K$, we get that 
\begin{equation}\label{convproba2}
 \Big(\frac{N_{i,T_\eps\wedge S_\eps^K}^{()}(T_\eps\wedge S_\eps^K)}{N_a(T_\eps\wedge S_\eps^K)}, i \geq 1\Big)-
 \Big(\frac{N_{i,\mathcal{T}_F^K}^{()}(\mathcal{T}_F^K)}{N_a(\mathcal{T}_F^K)}, i \geq 1\Big) \underset{K \to \infty}{\to} 0 \quad \text{in probability}.
\end{equation}We conclude the proof of Theorem \ref{theoprecise} point (2) by another application of Slutsky's theorem with \eqref{convloi1} and \eqref{convproba2}.

 \subsection{Proof of Theorem \ref{theoprecise} point (3)}

Introduce two positive constants 
$c_1=c_1(K,\eps)$ and $c_2=c_2(K,\eps)$ such that 
\begin{equation} \label{defc1c2} ( \eps^2 K^{\beta-c_1\eps}+1)e^{f_as_+(\eps)(1-\beta +c_2\eps)\log K/S_{aA}}=\eps^2K. \end{equation}
Notice that for $\eps$ small enough and $K$ large enough, we can choose $c_1$ and $c_2$ weakly dependent of 
$(K,\eps)$. More precisely, in this case, we get by using \eqref{defs} and \eqref{def_s_-s_+1} 
$$ c_1-c_2 \sim 2(1-\beta) \frac{C_{a,A}C_{A,a}}{S_{aA}C_{AA}} .$$ 
Let $i$ be smaller than $N_m(T_\eps^K \wedge S_\eps^K)$. Then we can define a martingale $M_i^+$ via
$$ M_i^{+}(t)= Z_i^{+}(T_i+t)e^{-f_a s_+(\eps)t}-1 , \quad t \geq 0.$$
An application of Doob's Maximal Inequality then gives
\begin{eqnarray*}
 \P\left( \underset{t \leq (1-\beta +c_2\eps)\log K/S_{aA}}{\sup}|M_i^{+}(t)|> \eps^2 K^{\beta-c_1\eps} \right)
&\leq &\frac{ \E \left[ \Big(M_i^{+}((1-\beta +c_2\eps)\log K/S_{aA}) \Big)^2 \right] }{( \eps^2 K^{\beta-c_1\eps})^2}\\
&\leq &\frac{2-s_+(\eps)}{s_+(\eps)}\eps^{-4} K^{-2\beta+2c_1\eps}.
\end{eqnarray*}
As a consequence, applying Definition \eqref{defc1c2} entails
$$
\P\left( \underset{t \leq (1-\beta +c_2\eps)\log K/S_{aA}}{\sup}Z_i^{+}(t)>\eps^2K\right)
\leq \frac{2f_a}{S_{aA}}\eps^{-4} K^{-2\beta+2c_1\eps}.
$$
Recalling point (2) of Theorem \ref{theoduration} and \eqref{compST} we get the existence of a positive constant $c$ 
such that for $\eps$ small enough
\begin{multline*}
\liminf_{K \to \infty}\P\left( \forall i \leq N_m(T_\eps^K \wedge S_\eps^K),  \underset{t \leq T_\eps^K \wedge S_\eps^K}{\sup}N_i(t)\leq\eps^2K \right) \\
\geq \liminf_{K \to \infty} \left(1- \frac{2f_a}{S_{aA}}\eps^{-4} K^{-2\beta+2c_1\eps}\right)^{c\log K K^{\beta}}-c\eps \geq 1-2c\eps.
\end{multline*}
This is equivalent to point (3) of Theorem \ref{theoprecise}.

\section{Proof ot Theorem \ref{theo_high_mut}} \label{proof_theo_high_mut}

Let us first check that the system \eqref{system_high_mut} admits no periodic orbit lying entirely in $\R^2_+$. To do this we apply Dulac's Theorem 
(see Theorem 7.12 in \cite{dumortier2006qualitative}). Introduce $\phi : (n_A,n_a) \in (\R^*_+)^2 \mapsto 1/n_An_a \in \R_+$. Then 
for every $(n_A,n_a) \in (\R^*_+)^2$, 
$$
 \partial_{n_A}(\phi \dot{n}_A)+  \partial_{n_a}(\phi \dot{n}_a)=- \Big( \frac{C_{A,A}n_A+C_{a,a}n_a}{n_An_a}+
 \frac{f_a \lambda^{aA}}{(n_A)^2}+ \frac{f_A \lambda^{Aa}}{(n_a)^2}\Big)<0,
$$
which ensures that \eqref{system_high_mut} has no periodic orbit on $\R^2_+$.

Recall that the parameters $( \rho_\alpha,\alpha \in \mathcal{A})$ have been defined in \eqref{defrhoalpha}, 
and that in this section, we assume that $\rho_\alpha >0, \alpha \in \mathcal{A}$.

We now focus on the fixed points of system \eqref{system_high_mut}. Finding a fixed point $(\tilde{n}_A, \tilde{n}_a) \in 
(\R_+^*)^2$ 
is equivalent to finding a positive real number $\rho$ such that $\tilde{n}_a= \rho \tilde{n}_A$. Indeed, if such a $\rho$ exists, 
then $\tilde{n}_A$ and $\tilde{n}_a$ are uniquely determined by
$$  \tilde{n}_A=(\rho_A+f_a\lambda^{aA}\rho)/(C_{A,A}+C_{A,a}\rho) \quad \text{and} \quad \tilde{n}_a= \rho \tilde{n}_A. $$
By substituting in system \eqref{system_high_mut} we see that 
such a $\rho$ is an admissible solution if and only if it is positive and satisfies:
\begin{equation*} \left\{\begin{array}{ll}
 \rho_A+f_a\lambda^{aA}\rho-(C_{A,A}+C_{A,a}\rho)n_A=0,\\
 \rho_a\rho+ f_A\lambda^{Aa} -(C_{a,A}\rho+C_{a,a}\rho^2)n_A=0,
\end{array}\right. 
\end{equation*}
and hence, 
$$ n_A=\frac{\rho_A+f_a\lambda^{aA}\rho}{C_{A,A}+C_{A,a}\rho}= 
\frac{\rho_a\rho+ f_A\lambda^{Aa}}{\rho(C_{a,A}+C_{a,a}\rho)} ,$$
or equivalently, 
\begin{multline*}
(f_a\lambda^{aA}C_{a,a})\rho^3+(f_a\lambda^{aA}C_{a,A}+\rho_AC_{a,a}- (\rho_a)C_{A,a})\rho^2 \\
+(\rho_AC_{a,A}- \rho_aC_{A,A}-f_A\lambda^{Aa}C_{A,a})\rho - f_A\lambda^{Aa}C_{A,A} =0. 
\end{multline*}
If we introduce $z= \rho - r$, where $r$ has been defined in \eqref{defr}, this sytem is equivalent to $z^3+pz+q=0$, where $p$ and $q$ have been defined in 
\eqref{defp} and \eqref{defq}, respectively. But the solutions of this equation are well known 
and can be computed by using Cardano's Method (see \cite{cardano1968ars}),
which gives the number of fixed points stated in Theorem \ref{theo_high_mut}. 

Let us now focus on the stability of the fixed points. 
The first step consists in evaluating the Jacobian matrix at a fixed point $\tilde{n}$:
\begin{eqnarray*} J(\tilde{n}):= \left( \begin{array}{cc}
                          \frac{\partial \dot{n}_A}{\partial n_A} &\frac{\partial \dot{n}_A}{\partial n_a}\\
                          \frac{\partial \dot{n}_a}{\partial n_A} & \frac{\partial \dot{n}_a}{\partial n_a}
                         \end{array} \right)_{|n=\tilde{n}}
                         &=&  \left( \begin{array}{cc}
                          \rho_A -2C_{A,A}\tilde{n}_A-C_{A,a}\tilde{n}_a  &- C_{A,a}\tilde{n}_A+f_a\lambda^{aA}\\
                          - C_{a,A}\tilde{n}_a+f_A\lambda^{Aa} &  \rho_a -2C_{a,a}\tilde{n}_a-C_{a,A}\tilde{n}_A
                         \end{array} \right).
\end{eqnarray*}
In particular, if $\tilde{n}=(0,0)$, 
\begin{eqnarray*} J(\tilde{n})=  \left( \begin{array}{cc}
                          \rho_A  &f_a\lambda^{aA}\\
                         f_A\lambda^{Aa} &  \rho_a 
                         \end{array} \right).
\end{eqnarray*}
As the trace of $J(0,0)$ is positive, this matrix has at least one positive eigenvalue. The sign of the second eigenvalue depends on the 
sign of the determinant of $J(0,0)$,
$$ Det (J(0,0))=\rho_A  \rho_a - f_a\lambda^{aA}f_A\lambda^{Aa} . $$
If $ Det (J(0,0))>0$ the two eigenvalues are positive and $(0,0)$ is a source (see Figure \ref{typesptsfixes}), and if 
$ Det (J(0,0))<0$ the two eigenvalues have opposite signs and $(0,0)$ is a saddle (see Figure \ref{typesptsfixes}).

If $\tilde{n}$ belongs to $(\R_+^*)^2$, the Jacobian matrix equivalently writes
\begin{eqnarray}\label{tracematrix} J(\tilde{n}):= \left( \begin{array}{cc}
                           -C_{A,A}\tilde{n}_A-{f_a\lambda^{aA}\tilde{n}_a}/{\tilde{n}_A} &- C_{A,a}\tilde{n}_A+f_a\lambda^{aA}\\
                          - C_{a,A}\tilde{n}_a+f_A\lambda^{Aa} &   -C_{a,a}\tilde{n}_a-{f_A\lambda^{Aa}\tilde{n}_A}/{\tilde{n}_a} 
                         \end{array} \right),
\end{eqnarray}
and has a negative trace.
Hence at least one of the eigenvalues is negative and $\tilde{n}$ cannot be a source.
To get the nature of the fixed points, we need to introduce the notions of sectorial decomposition and index for a fixed point.
We say that a vector field has
the finite sectorial decomposition property if it has the finite sectorial decomposition property at every isolated fixed point.
Roughly speaking this means that if we trace a circle with a small enough radius around an isolated fixed point, we can divide this circle 
in sectors where the vector field adopts one of the patterns described in figure \ref{secteurs} inside each sector (see \cite{dumortier2006qualitative} p 17-18
for a rigorous definition).

\begin{figure}[h]
    \centering
     \includegraphics[width=13cm,height=3cm]{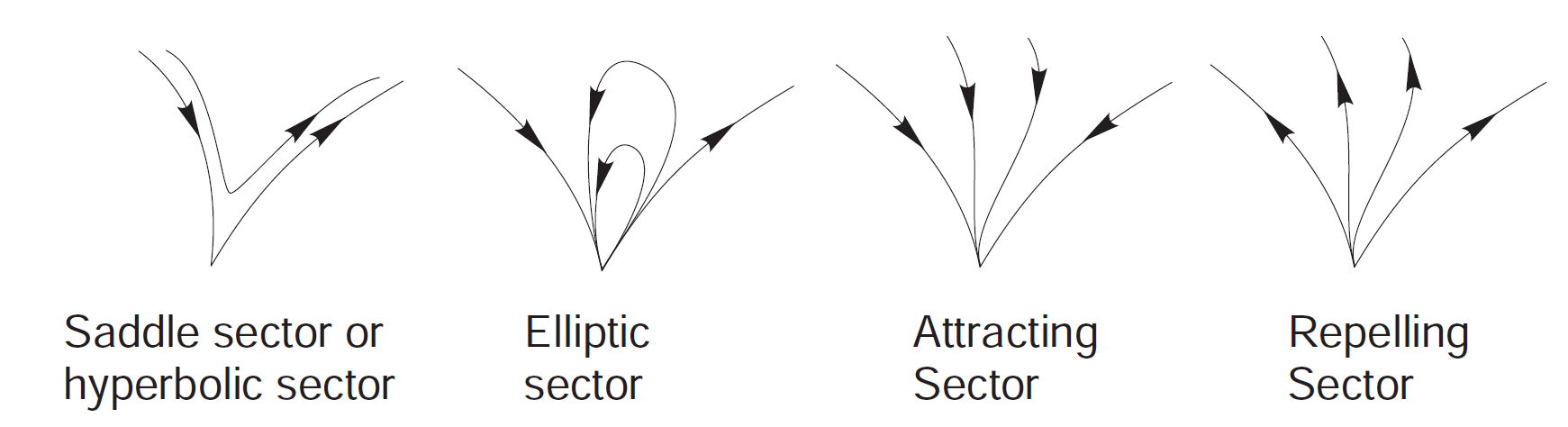}
     \caption{The different types of sectors around a fixed point (adapted from Figure 1.8 in \cite{dumortier2006qualitative}).}
   \label{secteurs}
\end{figure}

 When a vector field has the finite sectorial decomposition property, each isolated fixed point is given a characteristic couple 
$(e,h)$ which corresponds to the minimal number of elliptic and hyperbolic sectors, respectively, in a finite 
sectorial decomposition. 
The term minimal here has to be understood as follows: when a vector field has the finite sectorial decomposition property, the number of sectors around 
a fixed point is not unique, but if we gather adjacent sectors of the same type we get a 'minimal' decomposition with a given number of hyperbolic and elliptic 
sectors.
As in our case the vector field is polynomial and 
there may be only hyperbolic and semi-hyperbolic fixed points, the vector field has the finite sectorial decomposition property and we can apply results 
on the minimal sectorial decompositions.
In particular, combining Theorems 1.20 and 2.19 in \cite{dumortier2006qualitative}, we get that in our case (negative trace for the Jacobian matrix)
there are only three possible minimal sectorial decompositions (see Figure \ref{possibilites} for the corresponding phase portraits).
\begin{itemize}
 \item $(e,h)=(0,4)$
 \item $(e,h)=(0,0)$
 \item $(e,h)=(0,2)$
\end{itemize}

\begin{figure}[h]
    \centering
     \includegraphics[width=12cm,height=4cm]{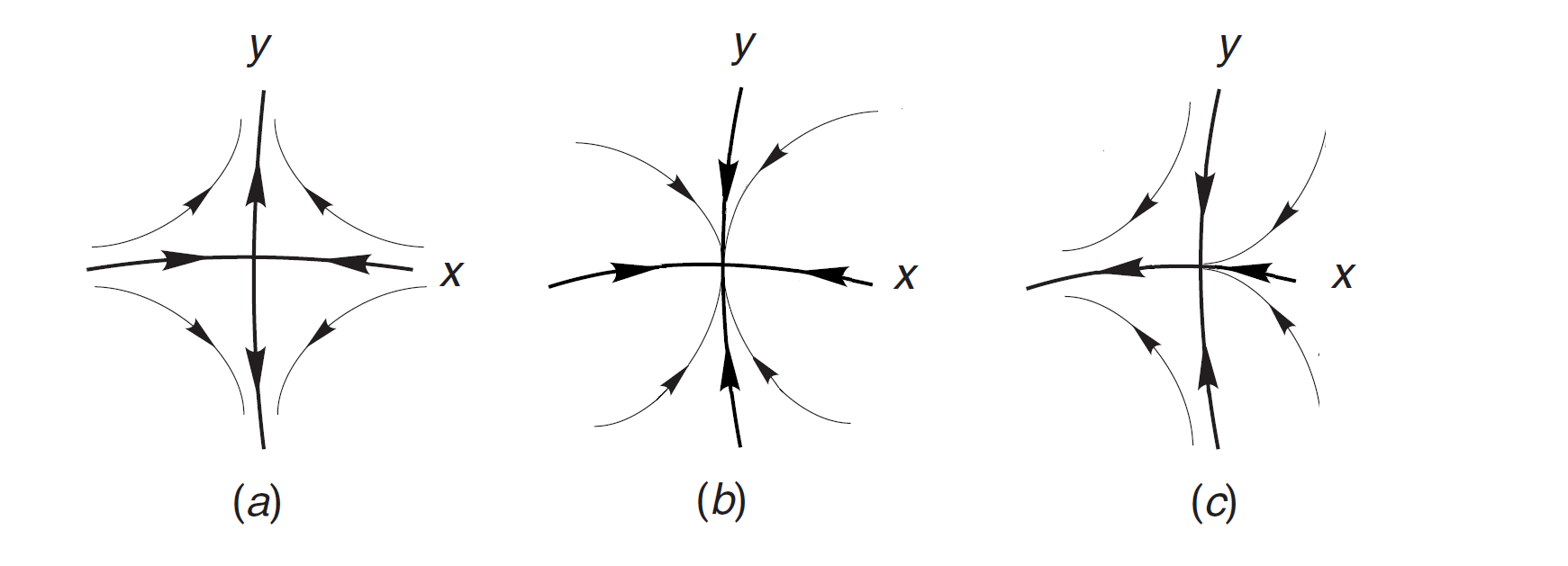}
     \caption{(a): $e=0,\ h=4$, (b): $e=0,\ h=0$, (c): $e=0,\ h=2$
     (adapted from Figure 2.13 in \cite{dumortier2006qualitative}).}
   \label{possibilites}
\end{figure}

Consider now a closed path $\sigma$ containing a finite number of isolated fixed points. Then we can define the index of this path, which is 
equal to 
the algebric 
number of turns of the vector field when moving along the closed path. This index only depends on the fixed points enclosed in the path 
(see Figure \ref{fig_theo_index} for its computation in our case).
As a consequence, we can define an index for each isolated fixed point, which corresponds to the index of a canonical closed path enclosing only this fixed point.
Applying the Poincaré Index Formula (see Proposition 6.32 in \cite{dumortier2006qualitative}) we get that the index $i(\tilde{n})$ of an hyperbolic or semi-hyperbolic fixed point $\tilde{n}$ is given by 
$$ i(\tilde{n})=\frac{e-h}{2}+1 .$$
Moreover, according to the Poincaré-Hopf Theorem (Theorem 6.26 in \cite{dumortier2006qualitative}), the index of a closed path is equal to the sum of the indices of the enclosed fixed points.
The two last properties stated will allow us to define the topological nature of the fixed points of system \eqref{system_high_mut}.
The first step consists in computing the index of a closed path containing all the fixed points in $(\R_+^*)^2$.
We choose a quarter circle centered at $0$, included in $\R_+^2$ and a radius large enough to enclose all the positive fixed points, 
and we exclude $0$ by substracting a quarter circle centered at $0$ with a small radius.
The index of this closed path does not depend on the parameters of system \eqref{system_high_mut} and is equal to $1$ as shown 
in Figure \ref{fig_theo_index}.
As a consequence, the sum of the indices of the positive fixed points is also equal to $1$, and we have the following possibilities:
\begin{enumerate}
 \item One fixed point: index $1$ (sink)
 \item Two fixed points: one with index $1$ (sink), one with index $0$ (saddle-node)
 \item Three fixed points: 
 \begin{itemize}
  \item two with index $0$ (saddle-nodes), one with index $1$ (sink)
  \item two with index $1$ (sinks), one with index $-1$ (saddle) 
 \end{itemize}
\end{enumerate}

\begin{figure}[h]
    \centering
     \includegraphics[width=6cm,height=6cm]{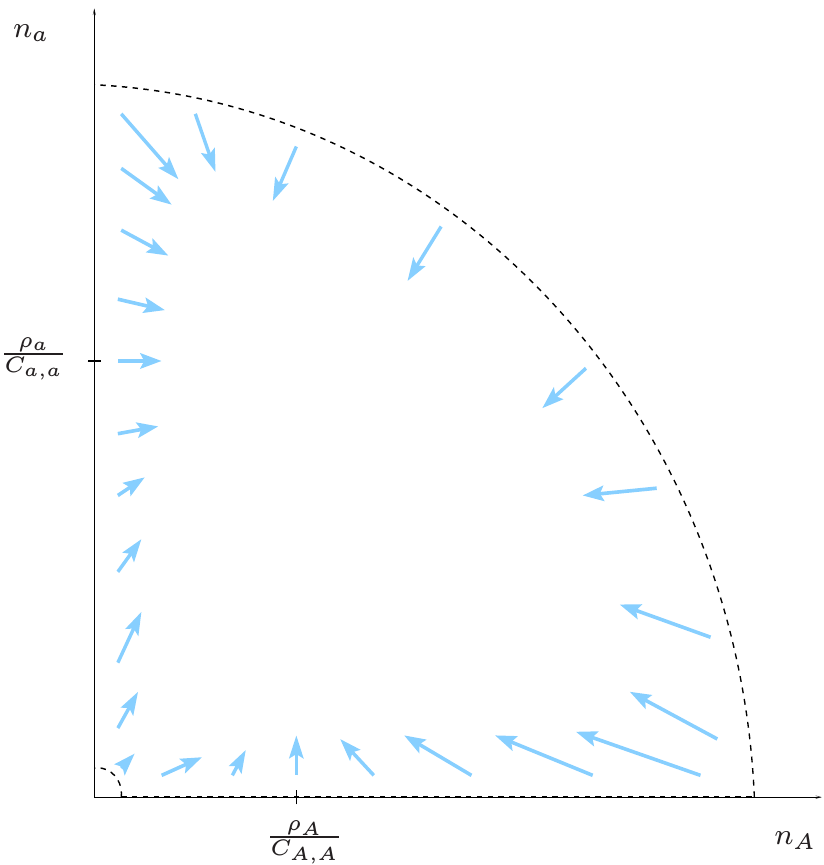}
     \caption{Computation of the index of a closed path containing all the positive fixed points.}
   \label{fig_theo_index}
\end{figure}

\section{Proof of Proposition \ref{propintuit}}\label{proof_prop_high_mut}

Recall Definition \eqref{defrhoalpha} and,
for the sake of readability, introduce the four following conditions
\begingroup
\setcounter{ass}{0} 
\renewcommand\theass{\Alph{ass}}
\begin{ass}\label{A}
$C_{A,a}\rho_A<C_{A,A}f_a\lambda^{aA}$
\end{ass}
\begin{ass}\label{B}
$C_{A,a}\rho_A\geq C_{A,A}f_a\lambda^{aA}$
\end{ass}
\begin{ass}\label{C}
$C_{a,A}\rho_a<C_{a,a}f_A\lambda^{Aa}$
\end{ass}
\begin{ass}\label{D}
$C_{a,A}\rho_a\geq C_{a,a}f_A\lambda^{Aa}$
\end{ass}
\endgroup

We are only interested in the fixed points of system \eqref{system_high_mut} which belong to $\R_+^2$ (no negative population densities).
The fixed points of system \eqref{system_high_mut} are the intersection points of two conics $\mathcal{C}_A$ and $\mathcal{C}_a$
whose equations are given by 
$$ \mathcal{C}_A: (\rho_A-C_{A,A}x-C_{A,a}y)x+ f_a\lambda^{aA}y= 0 ,$$
and
$$ \mathcal{C}_a: (\rho_a-C_{a,A}x-C_{a,a}y)y+ f_a\lambda^{Aa}x= 0, $$
in the canonical orthonormal reference frame $(\textbf{0},\mathbf{e}_x,\mathbf{e}_y)$.

We will make an orthonormal change of coordinates to get the reduced equation of the conic $\mathcal{C}_A$.
Let us first find the centre of $\mathcal{C}_A$, $(x_A,y_A)$. 
Let 
$$  x=u+x_A, \quad y= v + y_A . $$
If $x_A$ and $y_A$ satisfy, 
$$ x_A= \frac{f_a\lambda^{aA}}{C_{A,a}}, \quad y_A=\frac{C_{A,a}\rho_A-2C_{A,A}f_a\lambda^{aA}}{C_{A,a}^2} ,$$
then the equation of $\mathcal{C}_A$ writes, in the new coordinates $(u,v)$, 
$$ \frac{f_a\lambda^{aA}( C_{A,a}\rho_A-C_{A,A}f_a\lambda^{aA} )}{C_{A,a}^2}-C_{A,A}u^2-C_{A,a}uv=0 . $$
The second step consists in making a rotation to cancel the cross term. Let 
$$ e= u\cos\theta +v\sin\theta, \quad f=-u\sin\theta +v\cos\theta, $$
with $\theta$ in $[0,\pi)$. Then we get
$$ \frac{f_a\lambda^{aA}( C_{A,a}\rho_A-C_{A,A}f_a\lambda^{aA} )}{C_{A,a}^2}-C_{A,A}(e\cos\theta -f\sin\theta)^2-C_{A,a}(e\cos\theta -f\sin\theta)(e\sin\theta +f\cos\theta)=0, $$
and to cancel the cross term we need to choose $\theta$ such that
$$ \tan (2\theta)=\frac{C_{A,a}}{C_{A,A}}. $$
At this step we have not totally specified $\theta$ as two values in $[0,\pi)$ ($\theta \in [0,\pi/4)$ or $[\pi/2,3\pi/4)$) 
may satify the last equality.
However, the equation of $\mathcal{C}_A$ now writes
$$ \frac{f_a\lambda^{aA}( C_{A,a}\rho_A-C_{A,A}f_a\lambda^{aA} )}{C_{A,a}^2}+e^2(-C_{A,A}\cos^2\theta-\frac{C_{A,a}}{2}\sin(2\theta))+f^2(-C_{A,A}\sin^2\theta+\frac{C_{A,a}}{2}\sin(2\theta))=0 $$
Rearranging the terms, we get for this choice of $\theta$:
$$ e^2 \frac{C_{A,a}^2C_{A,A}\cos^2\theta}{f_a\lambda^{aA}( C_{A,a}\rho_A-C_{A,A}f_a\lambda^{aA} )\cos(2\theta)}
-f^2\frac{C_{A,a}^2C_{A,A}\sin^2\theta}{f_a\lambda^{aA}( C_{A,a}\rho_A-C_{A,A}f_a\lambda^{aA} )\cos(2\theta)}=1. $$
Hence to get the reduced equation we choose
$$ \theta_A:= \frac{1}{2}\arctan\Big( \frac{C_{A,a}}{C_{A,A}} \Big) +\frac{\pi}{2}\mathbf{1}_{\{ C_{A,a}\rho_A<C_{A,A}f_a\lambda^{aA} \}}. $$
In other words, $\theta_A=\arctan(C_{A,a}/C_{A,A})/2$ under Assumption \ref{B} and 
$\theta_A=\arctan(C_{A,a}/C_{A,A})/2+\pi/2$ under Assumption \ref{A}.
As a conclusion, in the orthonormal reference frame 
$$((x_A, y_A), \cos(\theta_A)\mathbf{e}_x +\sin(\theta_A)\mathbf{e}_y,
-\sin(\theta_A) \mathbf{e}_x +\cos(\theta_A)\mathbf{e}_y ),$$ $\mathcal{C}_A$ is an hyperbola with equation 
$$ \frac{x^2 C_{A,a}^2C_{A,A}\cos^2(\theta_A)}{f_a\lambda^{aA}( C_{A,a}\rho_A-C_{A,A}f_a\lambda^{aA} )\cos(2\theta_A)}
-\frac{y^2C_{A,a}^2C_{A,A}\sin^2(\theta_A)}{f_a\lambda^{aA}( C_{A,a}\rho_A-C_{A,A}f_a\lambda^{aA} )\cos(2\theta_A)}=1. $$
The asymptote "under the graph" is 
$$  y= - \frac{x}{|\tan (\theta_A)|}, $$
and the asymptote "over the graph" is 
$$  y=  \frac{x}{|\tan (\theta_A)|}. $$

\begin{figure}[h]
    \centering
 \begin{tabular}{cc}
     \includegraphics[width=5cm,height=5cm]{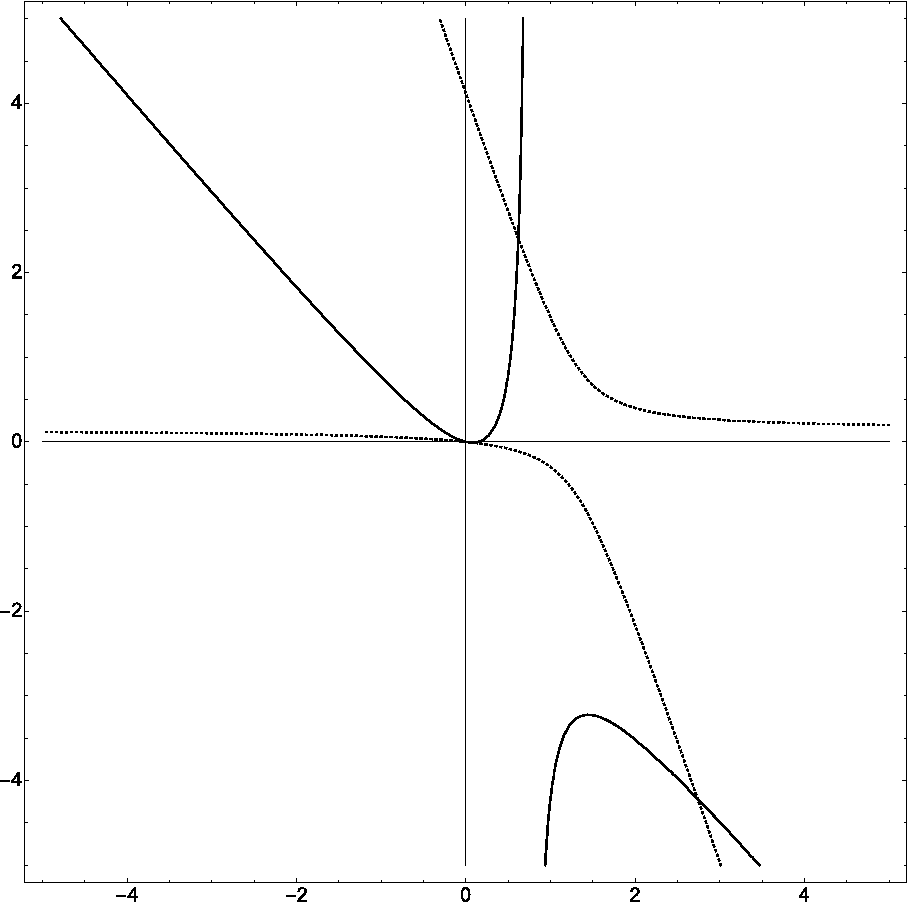} &
      \includegraphics[width=5cm,height=5cm]{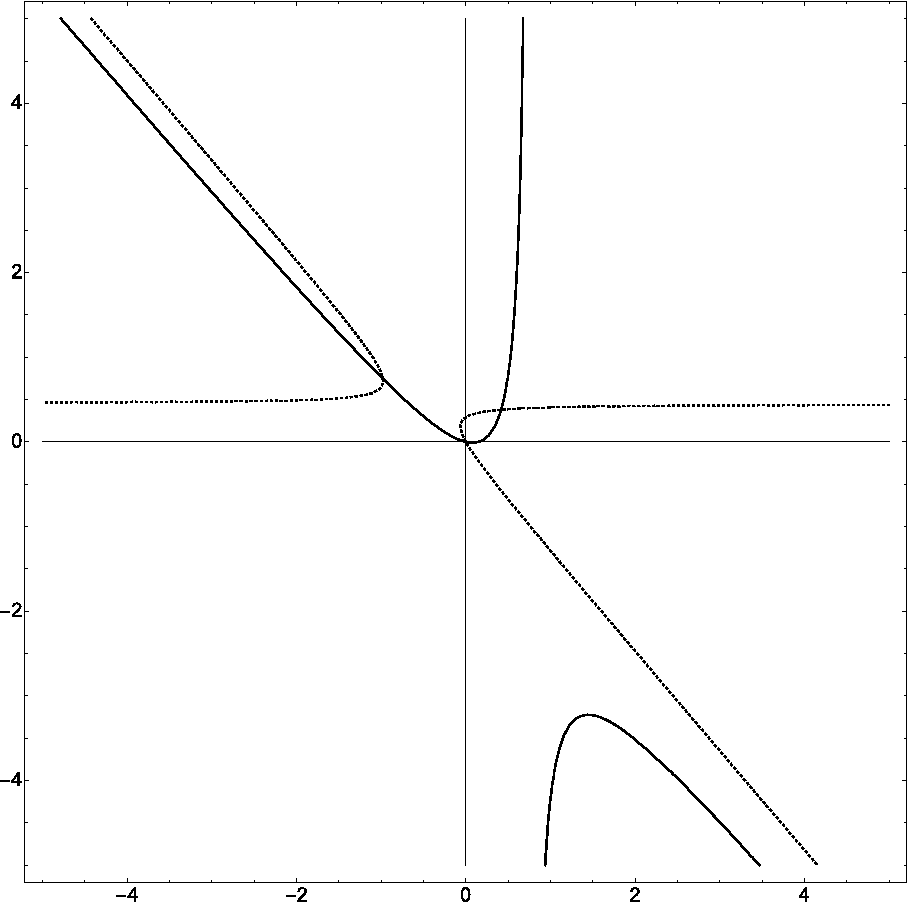}  \\
      (a) & (b) \\
    \end{tabular}
   \caption{Hyperbolas $\mathcal{C}_A$ and $\mathcal{C}_a$ under Assumptions \ref{A} and \ref{D} (a) and \ref{A} and \ref{C} (b). 
   The case \ref{B} and \ref{C} can be obtained by exchanging the two axes. The parameter values for these two figures are given in array \eqref{arrayvalues}.
   The conic $\mathcal{C}_a$ is represented with dotted line.}
   \label{fig1}
\end{figure}

Let us now go back to the canonical orthonormal reference frame $(\textbf{0},\mathbf{e}_x,\mathbf{e}_y)$.
Denote by $\mathfrak{b}_A$ the branch of the hyperbola $\mathcal{C}_A$ which has a part in $\R_+^2$.
Under Assumption \ref{A}, this branch crosses the $y$-axis in $(0,0)$ and $(\rho_A/C_{A,A},0)$, and its asymptote is the vertical line 
$x=x_A=f_a \lambda^{aA}/C_{A,A}$. Direct calculations give that the part of $\mathfrak{b}_A$ belonging to 
$\R_+^2$ satisfies $x \in [\rho_A/C_{A,A},x_A=f_a\lambda^{aA}/C_{A,a}]$, and makes an angle 
$$\theta \in \left[\arctan\left( \frac{\rho_A C_{A,A}}{f_a \lambda^{aA}C_{A,A}-\rho_AC_{A,a}} \right), \frac{\pi}{2}\right]$$
with the $x$-axis, which goes to $\pi/2$ when $x$ goes to $(f_a \lambda^{aA}/C_{A,A})^-$.

Similarly, under Assumption \ref{C}, the branch $\mathfrak{b}_a$ of the hyperbola $\mathcal{C}_a$ which has a part in $\R_+^2$
satisfies $y \in [\rho_a/C_{a,a},f_A\lambda^{Aa}/C_{a,A}]$, and makes an angle 
$$\theta \in \left[0, \frac{\pi}{2}-\arctan\left( \frac{\rho_a C_{a,a}}{f_A \lambda^{Aa}C_{a,a}-\rho_aC_{a,A}} \right)\right]$$
with the $x$-axis, which goes to $0$ when $y$ goes to $(f_A \lambda^{Aa}/C_{a,a})^-$.

Finally, we consider Assumption \ref{D}. In this case, $\mathfrak{b}_a$ crosses only once the $y$-axis, at $(0, \rho_a/C_{a,a})$, 
its asymptote is horizontal, and $\mathfrak{b}_a$ makes an angle 
$$\theta \in \left[\arctan\left( \frac{\rho_a C_{a,a}}{f_A \lambda^{Aa}C_{a,a}-\rho_aC_{a,A}} \right),0\right]$$
with the $x$-axis, which goes to $(f_A \lambda^{Aa}/C_{a,a})^+$.

The last statements imply that the conics $\mathcal{C}_A$ and $\mathcal{C}_a$ have only one fixed point on $(\R_+^*)^2$ under 
Assumptions $A$ and $C$ or $A$ and $D$. They also imply the relative position of this point with respect to 
$x_A, \rho_A/C_{A,A}, x_a$ and $\rho_a/C_{a,a}$ stated in Proposition \ref{propintuit}.\\

Let us now focus on the stability of this fixed point, $\tilde{n}$. 
According to \eqref{tracematrix} the trace of the Jacobian  matrix is negative at $\tilde{n}$,
and it is enough to check that the determinant is positive to show its stability:
\begin{multline} \label{expressiondet}
 Det(J(\tilde{n}))= \Big(C_{A,A}\tilde{n}_A+\frac{f_a\lambda^{aA}\tilde{n}_a}{\tilde{n}_A}\Big)\Big(C_{a,a}\tilde{n}_a+\frac{f_A\lambda^{Aa}\tilde{n}_A}{\tilde{n}_a} \Big)-
 (C_{a,A}\tilde{n}_a-f_A\lambda^{Aa})(C_{A,a}\tilde{n}_A-f_a\lambda^{aA})\\
 = (C_{A,A}C_{a,a}-C_{a,A}C_{A,a})\tilde{n}_A\tilde{n}_a + C_{a,a}\frac{f_a\lambda^{aA}}{\tilde{n}_A}(\tilde{n}_a)^2+ C_{A,A}\frac{f_A\lambda^{Aa}}{\tilde{n}_a}(\tilde{n}_A)^2\\
 + C_{a,A}\tilde{n}_af_a\lambda^{aA}+ C_{A,a}\tilde{n}_Af_A\lambda^{Aa}.
\end{multline}
By definition, $\tilde{n}$ satisfies
$$ C_{A,a}C_{a,A}\tilde{n}_A \tilde{n}_a = C_{A,a}(\rho_A\tilde{n}_A-C_{A,A}\tilde{n}_A^2+f_a\tilde{n}_a\lambda^{aA}) $$
Hence 
\begin{multline*}
 Det(J(\tilde{n})) = C_{A,A}C_{a,a}\tilde{n}_A\tilde{n}_a+ C_{A,a}\tilde{n}_A 
 (C_{A,A}\tilde{n}_A-\rho_A)
 + C_{a,a}\frac{f_a\lambda^{aA}}{\tilde{n}_A}(\tilde{n}_a)^2\\
 + C_{A,A}\frac{f_A\lambda^{Aa}}{\tilde{n}_a}(\tilde{n}_A)^2
 + C_{A,a}\tilde{n}_Af_A\lambda^{Aa}.
\end{multline*}
But we have shown that under Assumption \ref{A}, $C_{A,A}\tilde{n}_A>\rho_A$. This proves that $\tilde{n}$ is a stable fixed point.
We prove in the same way the stability of $\tilde{n}$ under Assumption \ref{C} by using
$$ C_{A,a}C_{a,A}\tilde{n}_A \tilde{n}_a = C_{a,A}(\rho_a\tilde{n}_a-C_{a,a}\tilde{n}_a^2+f_A\tilde{n}_A\lambda^{Aa}). $$

Besides, if Assumption \ref{A} does not hold but $C_{A,A}C_{a,a}-C_{a,A}C_{A,a}>0$, \eqref{expressiondet} shows that the determinant 
is positive and all fixed points in $(\R_+^*)^2$ are sinks. But 
according to the Poincare-Hopf Theorem, this is only possible if there is a unique fixed point in $(\R_+^*)^2$.
This ends the proof of Proposition \ref{propintuit}

\section{Proof of Proposition \ref{prop_high_mut2}} \label{proof_prop_conv}

The first step of the proof consists in studying the dynamical system \eqref{system_high_mut_bis} when $\lambda=0$.
This system has a unique stable equilibrium which is $\bar{n}=(0,(f_a-D_a)/C_{a,a})$ if $S_{Aa}<0<S_{aA}$, and 
$\bar{n}^{(aA)}=(n^{(aA)}_A,n^{(aA)}_a)$ (which has been defined in \eqref{defeqpospos})
if $S_{Aa}>0$ and $S_{aA}>0$. If we introduce the function $V : (\R_+^*)^2 \to \R$ 
$$ V(n_A,n_a)= n_A-\bar{n}_A^{(aA)}- \bar{n}_A^{(aA)} \ln \frac{n_A}{\bar{n}_A^{(aA)}}+n_a-\bar{n}_a^{(aA)}- \bar{n}_a^{(aA)} 
\ln \frac{n_a}{\bar{n}_a^{(aA)}}, $$
then $(V(n_A(t),n_a(t)), t \geq 0)$ is a Lyapunov function (see \cite{baigent2010lotka} pp 20-21 for a proof). As a consequence, the solutions of \eqref{system_high_mut_bis} with 
$\lambda=0$ converge to $\bar{n}$ as soon as 
$n_a(0)>0$ if $S_{Aa}<0<S_{aA}$, and as soon as
$n_\alpha(0)>0, \ \alpha \in \mathcal{A}$, if $S_{Aa}>0$ and $S_{aA}>0$. 
Moreover, applying Poincaré-Hopf
 Theorem as in the proof of Theorem \ref{theo_high_mut}, we get that $\bar{n}$ is an hyperbolic fixed point, and 
 more precisely a sink.
 The dynamical system \eqref{system_high_mut_bis} has a second fixed point, $(0,0)$, which is a source.
 We will now make a perturbation of this dynamical system by taking $\lambda$ small, and will show that it keeps the same 
 dynamical properties as when $\lambda$ is null. First, we prove that whatever the initial condition of the 
 dynamical system, the solutions enter a stable compact excluding the unstable fixed point $(0,0)$.
A direct calculation gives that as long as $n_A+n_a \geq 2  ((f_A- D_A) \vee (f_a-D_a))/\inf_{\alpha, \alpha' \in \{A,a\}^2}C_{\alpha,\alpha'}$,
 $$ \dot{n}_A+\dot{n}_a \leq - ((f_A- D_A) \vee (f_a-D_a))(n_A+n_a), $$
and as long as $n_A+n_a \leq  ((f_A-D_A)\wedge (f_a-D_a))/(2\sup_{\alpha, \alpha' \in \mathcal{A}^2}C_{\alpha,\alpha'})$,
 $$ \dot{n}_A+\dot{n}_a \geq  \frac{(f_A- D_A) \wedge (f_a-D_a)}{2}(n_A+n_a), $$
 and this independently of $\lambda$.
 We deduce that the compact
$$ \mathcal{C}:= \left\{ n \in \R_+^\mathcal{A}, 
 \frac{1}{2} \frac{(f_A-D_A)\wedge (f_a-D_a)}{\sup_{\alpha, \alpha' \in \{A,a\}^2}C_{\alpha,\alpha'}} \leq 
n_A+n_a \leq 2  \frac{(f_A- D_A) \vee (f_a-D_a)}{\inf_{\alpha, \alpha' \in \{A,a\}^2}C_{\alpha,\alpha'}} 
 \right\} $$
is reached in finite time by the solutions of the dynamical system \eqref{system_high_mut_bis} with a nonnull initial condition.
The two last inequalities also show that the set $\mathcal{C}$ is invariant under \eqref{system_high_mut_bis}.

Let us sum up our findings until now: if we introduce the function 
\begin{multline*} f(n_A,n_a,\lambda):= ((f_A(1- p\lambda)-D_A-C_{A,A}n_A-C_{A,a}n_a)n_A+ f_a\lambda n_a, \\
(f_a(1- \lambda)-D_a-C_{a,A}n_A-C_{a,a}n_a)n_a+ f_Ap \lambda n_A), \end{multline*}
then all the solutions of the dynamical systems defined by $((\dot{n}_A, \dot{n}_a)= f(n_A,n_a,\lambda), \lambda >0)$  with nonnull 
initial conditions reach in finite time the compact set $\mathcal{C}$.
Moreover, the fixed point $\tilde{n}$ of the system when $\lambda=0$ is hyperbolic (a sink) and attracts all the trajectories 
of the system with nonnull initial conditions.
Then inverse function theorem and the uniform continuity of the flows on the compact set $\mathcal{C}$ entail that
there exists a positive $\lambda_0(p)$ such that for every $\lambda \leq \lambda_0(p)$, the dynamical system defined by 
$(\dot{n}_A, \dot{n}_a)= f(n_A,n_a,\lambda)$ has a unique fixed point in $\mathcal{C}$, stable and globally attracting for solutions with 
nonnull initial conditions.
The first order approximation of this stable equilibrium with respect to $\lambda$ follows from standard calculations. 
This ends the proof of Proposition \ref{prop_high_mut2}.

\section{Illustration of Theorem \ref{theo_high_mut}} \label{sectionillustr}

In this section, we present some examples of the behaviour of system \eqref{system_high_mut}.
In Figures \ref{fig1} and \ref{fig2} we have drawn some examples of the conics $\mathcal{C}_A$ and $\mathcal{C}_a$. 
The parameters chosen for these illustrations are given in Table \eqref{arrayvalues}, and we now compute numerically the Jacobian matrices 
associated to the three fixed points in Figure \ref{fig2} (c).

\begin{figure}[h]
  \centering
   \begin{tabular}{ccc}
   \includegraphics[width=5cm,height=5cm]{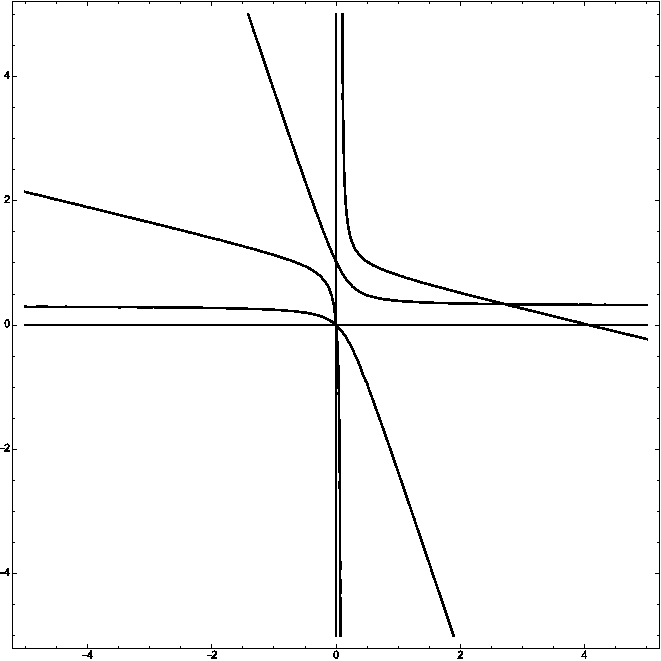}  &
      \includegraphics[width=5cm,height=5cm]{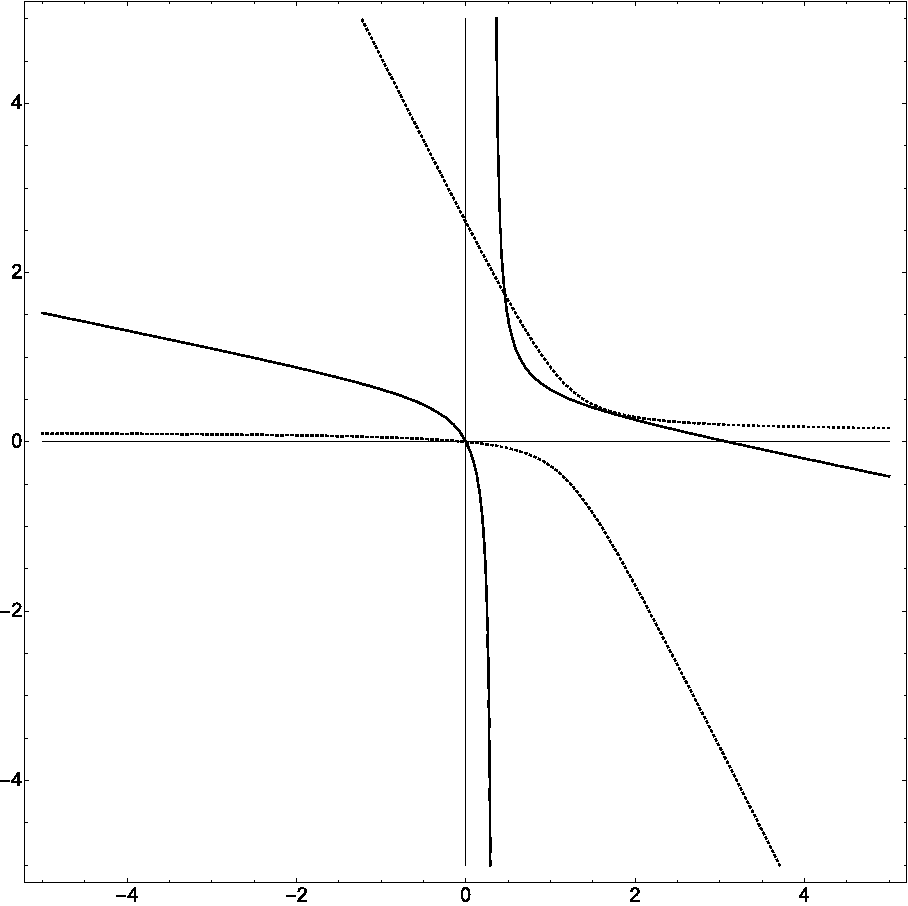} & 
      \includegraphics[width=5cm,height=5cm]{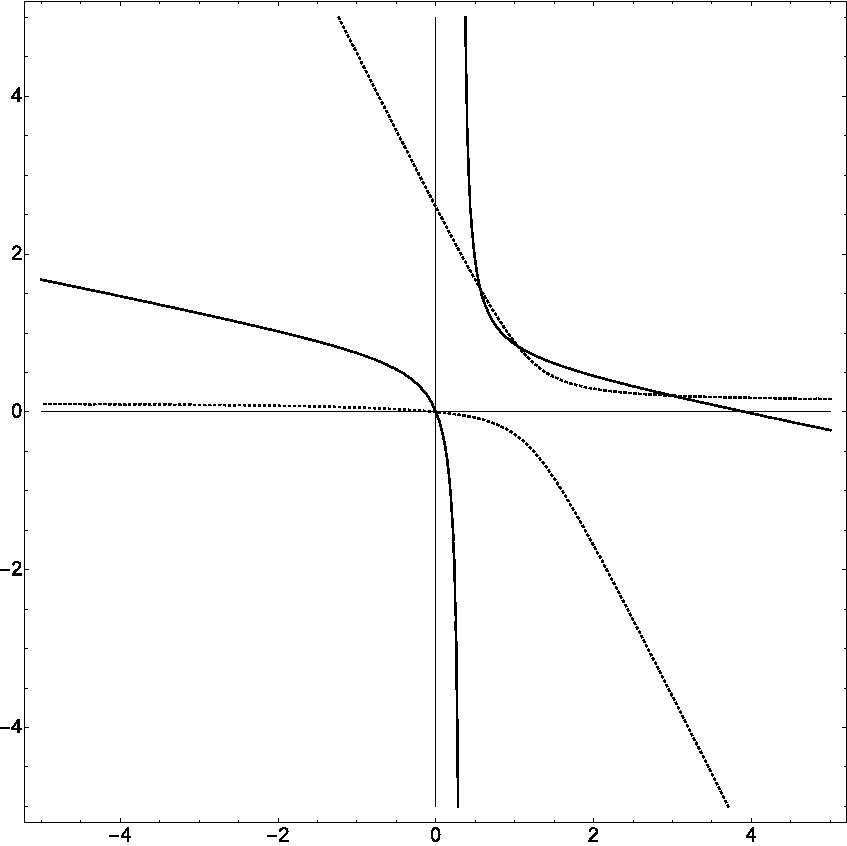} \\
      (a) & (b) & (c) \\
    \end{tabular} 
   \caption{Hyperbolas $\mathcal{C}_A$ and $\mathcal{C}_a$ under Assumptions \ref{B} and \ref{D}. In subfigure (a) they have one intersection point, two in 
   subfigure (b) and three in subfigure (c). The parameter values for these three figures are given in array \eqref{arrayvalues}.
   The conic $\mathcal{C}_2$ is represented with dotted line.}
   \label{fig2}
\end{figure}

\begin{equation}
  \label{arrayvalues}\begin{array}{|l|l|l|l|l|l|}
\hline
 & \text{Fig} \ \ref{fig1}.(a) & \text{Fig} \ \ref{fig1}.(b) & \text{Fig} \ \ref{fig2}.(a) & \text{Fig} \ \ref{fig2}.(b) & \text{Fig} \ \ref{fig2}.(c)  \\
\hline
\rho_A & 0.5 & 0.5  &  4.075 &  3.1 & 3.92 \\
\hline
C_{A,A} & 3.36 & 3.36 &  1 & 1 & 1\\
\hline
f_a & 5 & 5  &  2.19 & 5 & 5\\
\hline
\lambda^{aA} & 0.436 & 0.436 & 0.154  & 0.321 & 0.321\\
\hline
C_{A,a} & 2.86 & 2.86  & 4.17 & 5 & 5\\
\hline
\rho_a & 2.8 & 0.5  &  1.02 & 2.6 & 2.6\\
\hline
C_{a,a} & 0.68 & 1.71  & 1 & 1 & 1\\
\hline
f_A & 1.64 & 5  &  5 & 5 & 5\\
\hline
\lambda^{Aa} & 0.18 & 0.18  & 0.185 & 0.05 & 0.05\\
\hline
C_{a,A} & 2 & 2  &  3 & 2 & 2\\
\hline
\end{array}
\end{equation}

For the first fixed point in Figure \ref{fig2} (c) we get
$$ J( 0.562166, 1.56545 )=  \left( \begin{array}{cc}
                          - 5.0315705 & - 1.20583\\
                          - 2.8809  &  - 1.6552271
                         \end{array} \right)    , $$
and the eigenvalues of the jacobian matrix are $(- 5.8581153  , - 0.8286822 )$, which implies that this point is stable.

For the second fixed point in Figure \ref{fig2} (c) we get
$$ J( 1.03865, 0.83403 )=  \left( \begin{array}{cc}
                          - 2.3274558  &- 3.58825 \\ 
                          - 1.41806   & - 1.1453647
                         \end{array} \right)    , $$
and the eigenvalues of the jacobian matrix are $(- 4.0682955,0.5954749)$, which implies that this point is unstable.

For the third fixed point in Figure \ref{fig2} (c) we get
$$ J( 2.99029, 0.208302 )=  \left( \begin{array}{cc}
                          - 3.1020934 & - 3.58825\\         
                         - 0.166604  &  - 3.7971898 
                         \end{array} \right)    , $$
and the eigenvalues of the jacobian matrix are $( - 2.601935 , - 4.2973482  )$, which implies that this point is stable.

In Figure \ref{fig3}, we have drawn two trajectories of the system with intial conditions close to the unstable fixed point $(1.03865, 0.83403)$.

\begin{figure}[h]
    \centering
 \begin{tabular}{cc}
     \includegraphics[width=7cm,height=4cm]{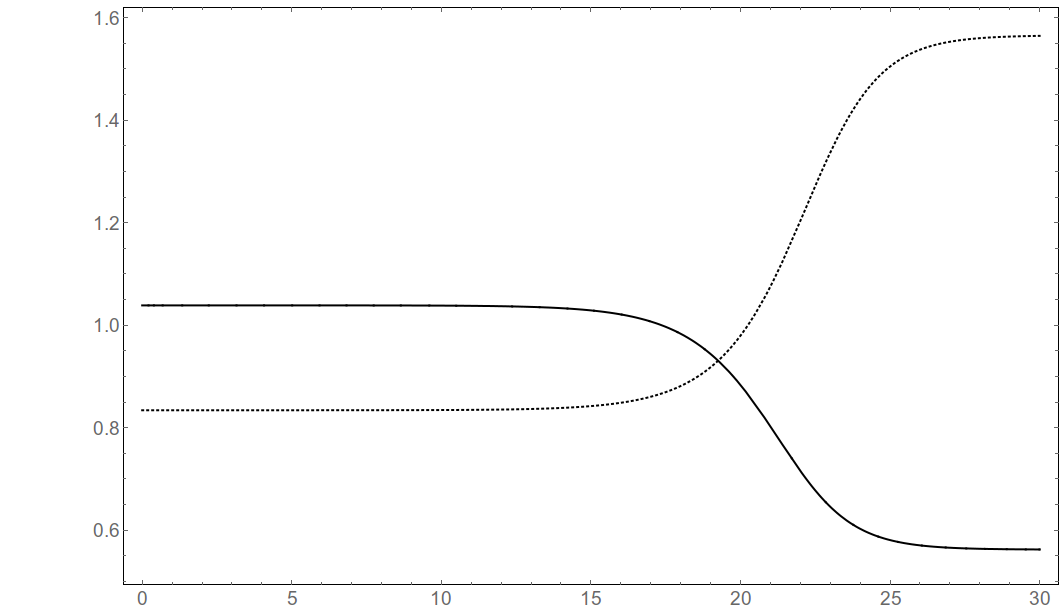} &
      \includegraphics[width=7cm,height=4cm]{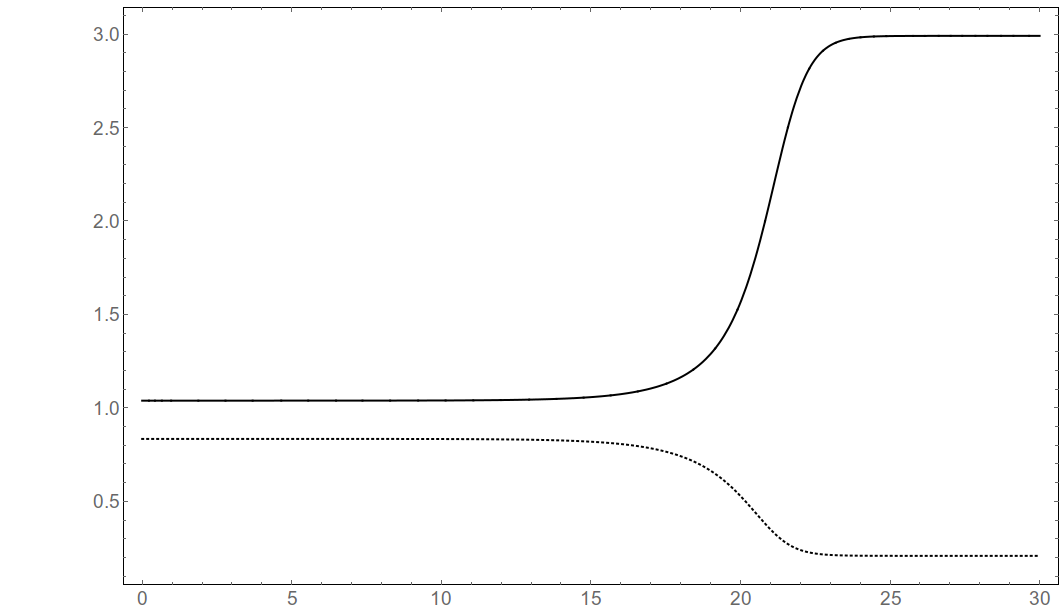}  \\
      (a) & (b) \\
    \end{tabular}
   \caption{Behaviour of the solution of the system \eqref{system_high_mut} when the initial condition is close to the unstable fixed point: 
   $(1.03865,0.83403)$ for the figure (a), where the solution converges to the first stable fixed point, and $(0.3866,0.83403)$ for the figure (b), where the solution converges to the second 
   stable fixed point. The $a$-population density is represented with dotted line.}
   \label{fig3}
\end{figure}

 \renewcommand\thesection{\Alph{section}}
\setcounter{section}{0}
 \renewcommand{\theequation}{\Alph{section}.\arabic{equation}}

 \section{Technical results}

In this last section, we recall some results on birth and death processes whose proofs can be found in Lemma 3.1 in
 \cite{schweinsberg2005random} and in \cite{MR2047480} p 109 and 112.

\begin{pro}
Let $Z=(Z_t)_{t \geq 0}$ be a birth and death process with individual birth and death rates $b$ and $d $. For $i \in \Z_+$, 
$T_i=\inf\{ t\geq 0, Z_t=i \}$ and $\P_i$ (resp. $\E_i$) is the law (resp. expectation) of $Z$ when $Z_0=i$. Then 
\begin{enumerate} 
 \item[$\bullet$] For $(i,j,k) \in \Z_+^3$ such that $j \in (i,k)$,
\begin{equation} \label{hitting_times1} \P_j(T_k<T_i)=\frac{1-(d/b)^{j-i}}{1-(d/b)^{k-i}} .\end{equation}
  \item[$\bullet$] If $d\neq b \in \R_+^*$, for every $i\in \Z_+$ and $t \geq 0$,
 \begin{equation} \label{ext_times} \P_{i}(T_0\leq t )= \Big( \frac{d(1-e^{(d-b)t})}{b-de^{(d-b)t}} \Big)^{i}.\end{equation}
 \item[$\bullet$] If $0<d<b$, on the non-extinction event of $Z$, which has a probability $1-(d/b)^{Z_0}$, the following convergence holds:
\begin{equation} \label{equi_hitting}  T_N/\log N \underset{N \to \infty}{\to} (b-d)^{-1}, \quad  a.s.  \end{equation}
\end{enumerate}
\end{pro}
 
The next Lemma quantifies the time spent by a birth and death process with logistic competition in a vicinity of its 
equilibrium size. It is stated in \cite{champagnat2006microscopic} Theorem 3(c).

\begin{lem}\label{Th3cChamp}
Let $b,d,c$ be in $\R_+^*$ such that $b-d>0$.
Denote by $(W_t,t \geq 0)$ a density dependent birth and death process with birth rate $bn$ and death rate
$(d+c n/K)n$, where $n \in \N$ is the current state of the process and $K \in \N$ is the carrying capacity.
Fix $0 < \eta_1 < (b - d)/c$ and $\eta_2 > 0$, and introduce the stopping time
$$\mathcal{S}_K = \inf \left\{t \geq  0 : W_t \notin
\left[\Big(\frac{b - d}{c}- \eta_1\Big)K,  \Big( \frac{b - d}{c}+\eta_2\Big)\right]\right\}.$$
Then, there exists $V > 0$ such that, for any compact subset $C$ of $](b - d)/c -
\eta_1 , (b - d)/c + \eta_2 [$,
\begin{equation}\label{temps_expo}\lim_{K \to \infty}
\underset{k/K \in C}{\sup} \P_k(\mathcal{S}_K< e^{KV} ) = 0.\end{equation}
\end{lem}

We end this section with a coupling of birth and death processes with close birth and death rates.

\begin{lem}\label{lemcompmart}
 Let $Z_1$ and $Z_2$ be two birth and death processes with intial state $1$ and respective individual birth and death rates $(b_1,d_1)$ and $(b_2,d_2)$ belonging to a common compact 
 set $D$ in $\R_+^2$. Then 
we can couple $Z_1$ and $Z_2$ 
 in such a way that 
 \begin{equation} \label{majvarquadraticdiff}
 \sup_{t \geq 0} \E\Big[ \Big( Z_1(t)e^{-(b_1-d_1)t}-Z_2(t)e^{-(b_2-d_2)t} \Big)^2 \Big]\leq c (|b_2-b_1|+|d_2-d_1|),  
 \end{equation}
where the positive constant $c$ only depends on $D$.
 \end{lem}

\begin{proof}
For the sake of simplicity, we assume in the proof than $b_1<b_2$ and $d_1<d_2$, but other cases can be treated similarly.
Let $B(ds,d\theta)$ and $D(ds,d\theta)$ be two independent Poisson random measures with intensity $ds d\theta$. We can construct the two processes $Z_1$ and $Z_2$ with respect to the measures 
 $B$ and $D$. 
 For $i \in \{1,2\}$, introduce
$$
  Z_i(t)= 1+ \int_0^t \int_{\R_+} \mathbf{1}_{\{ \theta \leq b_i Z_i(s^-) \}}B(ds,d\theta)
  -\int_0^t \int_{\R_+}\mathbf{1}_{\{ \theta \leq d_i Z_i(s^-) \}}D(ds,d\theta).
$$
We also introduce an auxiliary birth and death process, $Z_3$, with individual birth and death rates $(b_3,d_3)=(b_2,d_1)$, which will allow us to compare $Z_1$ and $Z_2$:
$$
  Z_3(t)= 1+ \int_0^t \int_{\R_+} \mathbf{1}_{\{ \theta \leq b_2 Z_3(s^-) \}}B(ds,d\theta)
  -\int_0^t \int_{\R_+}\mathbf{1}_{\{ \theta \leq d_1 Z_3(s^-) \}}D(ds,d\theta).
$$
As $b_1<b_2$ and $d_1<d_2$, we have the following almost sure inequalities:
\begin{equation}\label{ineqcoupl}
 Z_1(t)\leq Z_3(t)\quad \text{and} \quad  Z_2(t)\leq Z_3(t) \quad \text{a.s.}
\end{equation}
Applying It\^o's Formula yields that $M_i(t)= Z_i(t)e^{-(b_i-d_i)t}, i \in \{1,2,3\}$ are martingales, and we can express the quadratic variation of their differences.
Using \eqref{ineqcoupl} we get:
\begin{multline*}
 \langle M_3-M_1 \rangle_t = \int_0^t  \Big((b_1 +d_1) Z_1(s)(e^{-(b_1-d_1)s}-e^{-(b_2-d_1)s})^2\\
 + ( b_2 Z_3(s)- b_1 Z_1(s))e^{-2(b_2-d_1)s}+ ( d_1 Z_3(s)- d_1 Z_1(s))e^{-2(b_2-d_1)s}\Big)ds.
\end{multline*}
Then, by taking the expectation, we obtain
\begin{multline} \label{majM3M1}
 \E\Big[ (M_3-M_1)^2(t)\Big] = \int_0^t  \Big((b_1 +d_1) (e^{-(b_1-d_1)s}-2e^{-(b_2-d_1)s}+e^{-(2b_2-b_1-d_1)s})\\
 +  b_2e^{-(b_2-d_1)s} - b_1 e^{-(2b_2-d_1-b_1)s}+  d_1e^{-(b_2-d_1)s} - d_1 e^{-(2b_2-d_1-b_1)s}\Big)ds\\
 \leq (b_1 +d_1) \Big(\frac{1}{b_1-d_1}-\frac{2}{b_2-d_1}+\frac{1}{2b_2-b_1-d_1}\Big)\\
 +  \frac{b_2}{b_2-d_1} -  \frac{b_1}{2b_2-d_1-b_1}+  \frac{d_1}{b_2-d_1} - \frac{d_1}{2b_2-d_1-b_1}\leq c (b_2-b_1),
\end{multline}
where for the first inequality we have used that the square of a martingale is a 
submartingale, and in the last one the continuity of the functions involved. 
We obtain similarly
\begin{multline*}
 \langle M_3-M_2 \rangle_t = \int_0^t  \Big((b_2  Z_2(s)+d_2Z_2(s) \wedge d_1 Z_3(s) )(e^{-(b_2-d_1)s}-e^{-(b_2-d_2)s})^2\\
 + b_2(  Z_3(s)-  Z_2(s))e^{-2(b_2-d_1)s}\\
 + \mathbf{1}_{\{d_1 Z_3(s)< d_2Z_2(s) \}} ( d_2 Z_2(s)- d_1 Z_3(s))e^{-2(b_2-d_2)s}\\
 + \mathbf{1}_{\{d_1 Z_3(s)> d_2Z_2(s) \}} (  d_1 Z_3(s)-d_2 Z_2(s))e^{-2(b_2-d_1)s}\Big)ds.
\end{multline*}
But applying \eqref{ineqcoupl} we get that
\begin{multline*}
 \mathbf{1}_{\{d_1 Z_3(s)< d_2Z_2(s) \}} ( d_2 Z_2(s)- d_1 Z_3(s))\leq (d_2-d_1)Z_2(s)+d_2(Z_2(s)-  Z_3(s))\\
 \leq (d_2-d_1)Z_2(s)
\end{multline*}
and that 
$$  \mathbf{1}_{\{d_1 Z_3(s)> d_2Z_2(s) \}} (  d_1 Z_3(s)-d_2 Z_2(s))\leq  d_2(  Z_3(s)- Z_2(s)) ,$$
 which yields
\begin{multline*}
 \langle M_3-M_2 \rangle_t \leq \int_0^t  \Big((b_2  Z_2(s)+d_2Z_2(s) )(e^{-(b_2-d_1)s}-e^{-(b_2-d_2)s})^2\\
 + b_2(  Z_3(s)-  Z_2(s))e^{-2(b_2-d_1)s} + (d_2-d_1)Z_2(s)e^{-2(b_2-d_2)s}\\
 +  d_2(  Z_3(s)- Z_2(s))e^{-2(b_2-d_1)s}\Big)ds.
\end{multline*} 
 Taking the expectation and reasoning similarly as before give for every positive $t$,
 $$  \E\Big[ (M_3-M_2)^2(t)\Big]\leq c (d_2-d_1) . $$
 Using that for $a,b,c\geq 0$, $(a-c)^2 \leq 2(a-b)^2+2(b-c)^2 $ and adding \eqref{majM3M1} end the 
 proof of Lemma \ref{lemcompmart}
\end{proof}

{\bf Acknowledgements:} {\sl The author would like to thank Helmut Pitters for fruitful discussions at the beginning of this work.
She also wants to thank Jean-René Chazottes and Pierre Collet for advice and references on planar dynamical systems, as well as Pierre Recho for his help 
with the use of Mathematica, and two anonymous
reviewers for their careful reading of the paper, and several suggestions and improvements. 
This work was partially funded by the
Chair "Modélisation Mathémathique et Biodiversité" of Veolia
Environnement - Ecole Polytechnique - Museum National d'Histoire Naturelle - Fondation X and the
French national research agency ANR-11-BSV7- 013-03.}

\bibliographystyle{abbrv}

\end{document}